\documentclass[10pt]{amsart}
\usepackage{amssymb,a4wide}
\usepackage{amsthm}
\usepackage{graphicx}
\usepackage[T2A]{fontenc}
\usepackage[latin1]{inputenc}
\usepackage{mathdots}
\usepackage{tensor}
\usepackage{amscd}

\newcommand{\hnabla}[0]{\tensor[^h]{\nabla}{}}

\newcommand{\weg}[1]{}

\newcommand{\Id}{\mathrm{Id}}

\theoremstyle{plain}
\newtheorem{thm}{Theorem}
\newtheorem*{thm*}{Theorem}
\newtheorem{lem}{Lemma}

\newtheorem{cor}{Corollary}

\theoremstyle{definition}

\theoremstyle{remark}
\newtheorem{rem}{Remark}

\title[Conification construction and its applications]{Conification construction for K\"ahler manifolds  and its application in c-projective geometry}
  \author{Vladimir S. Matveev and Stefan Rosemann}
\address{Institute of Mathematics, Friedrich-Schiller-Universit\"at Jena, Jena, Germany.}
\email{vladimir.matveev@uni-jena.de}
\address{Institute of Mathematics, Friedrich-Schiller-Universit\"at Jena, Jena, Germany.}
\email{stefan.rosemann@uni-jena.de}

\begin{document}

\begin{abstract}
Two K\"ahler metrics on one complex manifold are said to be c-projectively equivalent if their $J$-planar curves, i.e., curves defined by the property that their acceleration is complex proportional to their velocity, coincide.
The degree of mobility of a K\"ahler metric is the dimension of the space of metrics that are c-projectively equivalent to it. We give the  list  
 of all possible  values  of the degree of mobility of simply connected $2n$-dimensional Riemannian K\"ahler manifolds. We also describe all such values under the additional  assumption that   the metric is Einstein. As an application, we describe all  possible dimensions of the space of essential c-projective vector fields of K\"ahler and  Kähler-Einstein Riemannian  metrics.   We also  show that two c-projectively equivalent K\"ahler Einstein metrics  (of arbitrary signature)
on a closed manifold have constant holomorphic  curvature or  are affinely equivalent.
\end{abstract}

\maketitle

\section{Introduction}

The conification construction will be recalled in \S \ref{sec:construction}. It is a  special case of a local construction  from  \cite{ACM} that given a $2n$-dimensional Kähler manifold $(M, g, J)$ (of arbitrary signature) produces a $(2n+2)$-dimensional Kähler manifold. If we apply this construction to the Fubini-Study metric $g_{FS}$, we obtain the flat metric.
 
There are many similar and more general constructions that were described and successfully applied in Kähler geometry before, for example the Calabi construction or the interplay between Sasakian and Kähler manifolds. 

The main results of the paper are  applications of the construction in the theory of c-projectively equivalent metrics, let us explain what this theory is about.

Let $(M,g,J)$ be a K\"ahler manifold of real dimension $2n\geq 4$ with Levi-Civita connection $\nabla$. A regular curve $\gamma:I\rightarrow M$ is called \emph{$J$-planar} for $(g,J)$ if
\begin{align}
\left(\nabla_{\dot{\gamma}}\dot{\gamma}\right)\wedge\dot{\gamma}\wedge J\dot{\gamma}=0\label{eq:hplanar}
\end{align}
holds at each point of the curve.  The condition \eqref{eq:hplanar} can 
 equivalently be rewritten as follows: there exist smooth functions $\alpha(t), \beta(t)$ such that $$\nabla_{\dot{\gamma}(t)}\dot{\gamma}(t)= \alpha(t) \cdot \dot \gamma(t) + \beta(t) \cdot J\dot\gamma(t) .$$
Every geodesic is evidently a $J$-planar curve.  The set of $J$-planar curves is geometrically  a much bigger set  of curves than the set  of  geodesics. For example in every point and in every direction there exist infinitely many  geometrically different $J$-planar curves, see figure \ref{manyhplanar}. 

\begin{figure}
  \includegraphics[width=.3\textwidth]{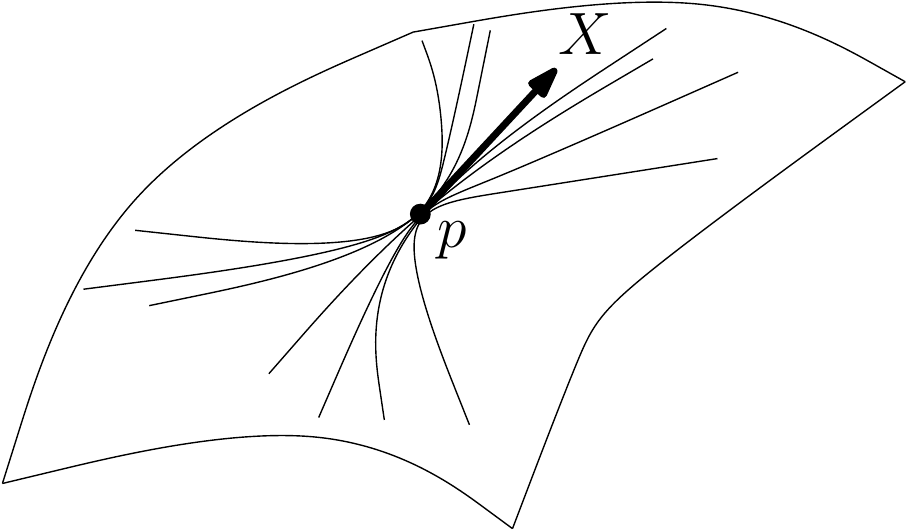}
  \caption{Given $p\in M$, $X\in T_p M$, there are infinitely many $J$-planar curves $\gamma$ such that $\gamma(0)=p,\dot{\gamma}(0)=X$.}\label{manyhplanar}
\end{figure}

Two metrics $g,\tilde{g}$ on the complex manifold $(M,J)$  which are K\"ahler  w.r.t. the complex structure $J$ are called \emph{c-projectively equivalent} if  their $J$-planar curves coincide.  A trivial example of c-projectively equivalent metrics is when the metric $\tilde{g}$ is proportional to $g$ with a constant coefficient.  Another trivial (in the sense it is relatively easy to treat it at least in the Riemannian, i.e., positively-definite case)  example is  when the metrics are \textit{affinely equivalent}, that is when  their Levi-Civita connections coincide. If these metrics are Kähler w.r.t.  the same complex structure, they are of course c-projectively equivalent since the  equation \eqref{eq:hplanar} defining $J$-planar curves involves the connection and the complex structure only. 

The theory of c-projectively equivalent K\"ahler metrics  is a classical one. It was started  in the 50th in  Otsuki et al   \cite{Otsuki1954}  and for a certain period of time was one of the main research directions of the japanese (Obata, Yano) and soviet (Sinjukov, Mikes) differential geometric schools, see the survey \cite{Mikes} or  the books \cite{Yanobook,Sinjukov} for  an overview  of the classical results.  In the recent time, the theory of c-projectively  equivalent metrics has a revival. A number of new methods appeared within or were applied in the c-projective setting and classical conjectures were solved.  Moreover, c-projectively equivalent metrics independently 
appeared under the name Hamiltonian $2$-forms, see \cite{ApostolovI,ApostolovII,ApostolovIII,ApostolovIV}, which are also closely related to conformal Killing or twistor $2$-forms studied in \cite{moruianusemmelmann,semmelmann}. These relations will be explained in more detail in Section \ref{sec:reltoothers}. The c-projectively equivalent metrics also play a role in the theory of (finitely-dimensional) integrable systems \cite{Kiyo1997} where they are closely related to the so called Kähler-Liouville metrics, see  \cite{Kiyohara2010}. 
 
 \begin{rem} Most classical  sources  use the name ``h-projective'' or ``holomorphically-projective'' for what we call ``c-projective''  in our paper. We also used ``h-projective'' in our previous publications \cite{FKMR, MatRos}. 
  Recently a group of geometers studying c-projective geometry  from different viewpoints decided to change the name from h-projective to c-projective, since a c-projective change of connections, though being complex in the natural sense,  is generically not holomorphic. The prefix  ``c-'' is chosen to be reminiscent of ``complex-'' but is not supposed
to be pronounced nor regarded as such since ``complex projective'' is
already used differently in the literature.
 \end{rem} 

As we recall in Section \ref{sec:basics}, the set of metrics c-projectively equivalent to a given one (say, $g$)
is in one-to-one correspondence with the set of  nondegenerate hermitian symmetric $(0,2)$-tensors $A$ satisfying the equation
\begin{align}
(\nabla_Z A)(X,Y)=g(Z,X)\lambda(Y)+g(Z,Y)\lambda(X)+\omega(Z,X)\lambda(JY)+\omega(Z,Y)\lambda(JX)\label{eq:mainA}
\end{align}
for all $X,Y,Z\in TM$. Here $\lambda$ is a $1$-form which is easily seen to be equal to $\lambda=\tfrac{1}{4}d\mathrm{trace}\,A$ (here, $A$ is viewed as a $(1,1)$-tensor by ``raising one index'' with the metric). Since this equation is linear, the space of its solutions is a linear vector space. Its dimension is called
the \emph{degree of mobility} of $(g,J)$ and will be denoted by $D(g,J)$. Locally, $D(g,J)$ coincides with the dimension of the set (equipped with its natural topology) of metrics c-projectively 
equivalent to $g$.

Our main result is the list of all  possible degrees of mobility of Riemannian Kähler metrics  on simply connected manifolds  (in what follows we always assume that  simply connectedness implies   connectedness). The degree of mobility is always  $\ge 1$ since the metrics of the form $\textrm{const}\cdot \  g$ provide a one-parameter family of metrics c-projectively equivalent to $g$. One can show that for a generic  metric the degree of mobility is  precisely $1$. This statement, though known in folklore, is up to our knowledge nowhere published. Let us mention therefore that by  \cite{Mikes} irreducible symmetric Riemannian Kähler spaces of  nonconstant holomorphic curvature have degree of mobility precisely $1$.  Using this result one could show, similar to \cite[\S 3.1]{hall}, that for any (local) Kähler metric $g$ and every $\epsilon>0$ there exists a Kähler  metric $g'$  that is $\varepsilon-$close to $g$ in the $C^\infty$ topology such  that there exists $\epsilon'>0$ such that for every metric $g''$ that is $\varepsilon'-$close to $g'$ in the $C^\infty$ topology     the  degree of mobility of $g''$ is precisely  $1$. \weg{
 The maximal degree of mobility is $(n+1)^2$ which is attained on  simply connected  manifolds of constant holomorphic curvature. }

\begin{thm}\label{thm:degree}
Let $(M,g,J)$ be a simply connected Riemannian K\"ahler manifold of real dimension $2n\geq 4$. Suppose at least one metric c-projectively equivalent to $g$ is not affinely equivalent to it. 

Then, the degree of mobility $D(g,J)$ belongs to the following  list:
\begin{itemize}
\item $2$,
\item $k^2+\ell$, where $k=0,...,n-1$ and $\ell=1,...,\big[\tfrac{n+1-k}{2}\big]$.
\item $(n+1)^2$,
\end{itemize}

Moreover, every  value  from  this  list that is greater than or equal to $2$ is  the degree of mobility of a certain simply connected  $2n$-dimensional K\"ahler manifold $(M, g,J)$ such that there exists a c-projectively but not affinely equivalent metric $\tilde g$.
\end{thm}
In the theorem above, $\big[\ .\ \big]$ denotes  the integer part. The condition    $D(g,J)\ge 2$ is due to the  assumption   that there exists a metric that is c-projectively but not affinely equivalent (and therefore not proportional) to $g$. 

\begin{figure}
  \includegraphics[width=.7\textwidth]{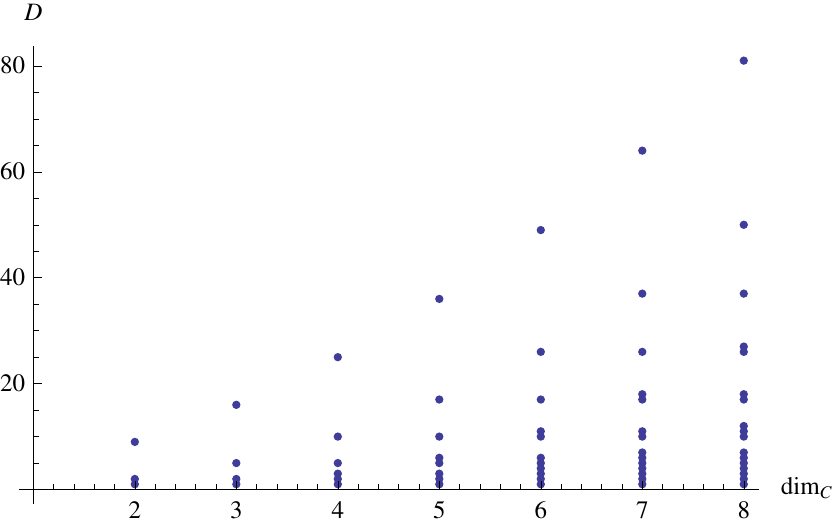}
  \caption{The possible values for the degree of mobility $D$ for $2\leq\tfrac{1}{2}\mathrm{dim} M\leq  8$. }\label{1}
\end{figure}

If all metrics that are c-projectively equivalent to $g$ are affinely equivalent to $g$, then  it is not a big deal to obtain in the Riemannian situation 
the list of all possible degrees of mobility of such  metrics on simply connected manifolds by  using the same circle of ideas as in the proof of Theorem \ref{thm:degree} (see also Section \ref{sec:ideas} below): it is 
$$
\{k^2 + \ell \mid  k\le n-1, \  1 \le \ell \le   {n-k}\}\bigcup \{   n^2\}. 
$$ 

Special cases of Theorem \ref{thm:degree} were known before. 
It is a classical result (see e.g. \cite[Proposition 4]{ApostolovI} or \cite[Chap. V, \S 3]{Sinjukov}) 
 that the maximum value $D(g,J)=(n+1)^2$ implies that the metric has constant holomorphic curvature and is attained on simply connected manifolds of  constant holomorphic curvature.  
It was also previously known (\cite[Proposition 10]{ApostolovII} and \cite[Lemma 6]{FKMR}) that in the   case when the dimension is $2n=4$ the degree of mobility (on a simply connected manifold) takes the values  $1,2,9$ only.  We also  see  that the submaximal  degree of mobility  is $(n-1)^2+ 1 = n^2 - 2n +2$. This value was also known before, see \cite[\S 1.2]{Mikes}, though we did not find a place where this statement was proved.

Under the additional assumption that the manifold $M$ is closed, the analog of  Theorem \ref{thm:degree}  is also essentially known and the list of possible degrees of mobility 
 is $\{1,2, (n+1)^2\}$. Indeed, as it was shown in \cite{FKMR}, if $D(g,J)\geq 3$, $(M,g,J)$ must be equal to $(\mathbb{C}P(n),\mbox{const}\cdot g_{FS},J_{standard})$ and therefore its degree of mobility is  $(n+1)^2$. On the other hand, there are   many examples (constructed  in \cite{ApostolovII,ApostolovIII,ApostolovIV} or \cite{Kiyohara2010}) of closed K\"ahler manifolds different from $(\mathbb{C}P(n),\mbox{const}\cdot g_{FS},J_{standard})$ admitting c-projectively equivalent but  not affinely equivalent Riemannian  K\"ahler metrics. Thus, on closed K\"ahler manifolds (and here the assumption that the manifold is simply connected is not important), 
 $D(g,J)$  takes  the values $1,2$ and $(n+1)^2$ only.

\subsection{The dimension of the space  of essential c-projective vector fields}

A (possibly, local) diffeomorphism  $f:M\rightarrow M$ of a K\"ahler manifold $(M,g,J)$ is called \emph{c-projective transformation} if it is \emph{holomorphic} (i.e., preserves $J$) 
and if it  sends  $J$-planar curves to $J$-planar curves. A clearly equivalent requirement is  that  the pullback $f^*g$ is c-projectively equivalent to $g$. 
A c-projective transformation is called \emph{essential}, if it is not an isometry. A vector field is called a \emph{c-projective vector field} if its (locally defined) flow consists of c-projective transformations; we call it essential if it is not a Killing vector field.

 For a given K\"ahler structure $(g,J)$, let $\mathfrak{c}(g,J)$ and  $\mathfrak{i}(g,J)$ denote the Lie algebras of c-projective and holomorphic Killing vector fields respectively. Both are linear  vector spaces and $\mathfrak{c}(g,J) \supseteq \mathfrak{i}(g,J)$. 
  The quotient vector space  $\mathfrak{c}(g,J)/\mathfrak{i}(g,J)$ will be called the \emph{space of essential c-projective vector fields}. Its dimension is 
  $\dim(\mathfrak{c}(g,J))- \dim(\mathfrak{i}(g,J))$.  In the proof of Theorem \ref{thm:hprotrafo} it will be clear that under the assumption that the degree of mobility is $\ge 3$ this vector space could be (canonically) viewed as  a subspace of  $\mathfrak{c}(g,J)$. It is not a subalgebra though: typically the commutator of two vector fields from this space is a nontrivial Killing vector field. 
  Let us also mention  that the number $\dim(\mathfrak{c}(g,J))- \dim(\mathfrak{i}(g,J))$ remains the same if we replace $g$ by a c-projectively equivalent metric $\tilde g$.  
  
\begin{thm}\label{thm:hprotrafo}
Let $(M,g,J)$ be a simply connected Riemannian K\"ahler manifold of real dimension $2n\geq 4$. Suppose at least one metric c-projectively equivalent to $g$ is not affinely equivalent to it. 

Then, the   dimension of the space $\mathfrak{c}(g,J)/\mathfrak{i}(g,J)$ is 
\begin{itemize}
\item $0,1$, or  
\item $k^2+\ell-1$, where $k=0,...,n-1$ and $\ell=1,...,\big[\tfrac{n+1-k}{2}\big]$, or 
\item $(n+1)^2-1$.
\end{itemize}

Moreover, each of the values of the above list is equal to the number $\mathrm{dim}(\mathfrak{c}(g,J)/\mathfrak{i}(g,J))$ for a certain $2n$-dimensional simply connected 
Riemanian K\"ahler manifold  $(M,g,J)$ such that there exists a metric that is c-projectively but not affinely equivalent to $g$.
\end{thm}

Under the assumption that the manifold is closed, the analog of  Theorem \ref{thm:hprotrafo} is again known and is due to \cite{MatRos} where  the classical Yano-Obata conjecture is proved; this conjecture implies   that  on any  closed Riemannian K\"ahler manifold   $\mathrm{dim}(\mathfrak{c}(g,J))=\dim(\mathfrak{i}(g,J))$ unless the manifold is $(\mathbb{C}P(n),\mbox{const}\cdot g_{FS},J_{standard})$ .

\subsection{ Einstein metrics}
Our next group of results concerns Kähler-Einstein metrics.  Note that Einstein metrics  play a special  important role in the c-projective geometry since they are closely related to the normal sections of the so-called  prolongation connection of the metrizability equation, see \cite{EMN}.  

The analog of Theorem \ref{thm:degree} under the additional assumption that the metric is Einstein is 
\begin{thm}\label{thm:degreeeinstein}
Let $(M,g,J)$ be a simply connected Riemannian K\"ahler-Einstein manifold of real dimension $2n\geq 4$. Assume at least one metric c-projectively equivalent to $g$ is not affinely equivalent to it. Then, the degree of mobility $D(g,J)$ is equal to one of the following numbers:
\begin{itemize}
\item $2$,
\item $k^2+\ell$, where $k=0,...,n-2$ and $\ell=1,...,\big[\tfrac{n+1-k}{3}\big]$,
\item $(n+1)^2$. 
\end{itemize}

Conversely, each of the numbers of the above list that is greater than or equal to $2$  is  the degree of mobility of a certain 
simply connected  Riemannian  K\"ahler-Einstein manifold $(M^{2n},g,J)$ admitting a metric $\tilde g$ that is c-projectively equivalent but not affinely equivalent to $g$.
\end{thm}

The analog of Theorem \ref{thm:hprotrafo} under the additional assumption that the metric is Einstein looks as follows:
\begin{thm}\label{thm:hprotrafoeinstein}
Let $(M,g,J)$ be a simply connected Riemannian K\"ahler-Einstein  manifold of real dimension $2n\geq 4$. Suppose at least one metric c-projectively equivalent to $g$ is not affinely equivalent to it. 

Then, the   dimension of the space $\mathfrak{c}(g,J)/\mathfrak{i}(g,J)$ is 
\begin{itemize}
\item $0,1$, or  
\item $k^2+\ell-1$, where $k=0,...,n-2$ and $\ell=1,...,\big[\tfrac{n+1-k}{3}\big]$, or 
\item $(n+1)^2-1$.
\end{itemize}

Moreover, each of the values of the above list is equal to  $\mathrm{dim}(\mathfrak{c}(g,J)/\mathfrak{i}(g,J))$ for a certain $2n\ge 4$-dimensional simply connected 
Riemannian K\"ahler-Einstein manifold  $(M,g,J)$ such that there exists a metric that is c-projectively but not affinely equivalent to $g$.
\end{thm}

We see that the main  difference in the lists of Theorems \ref{thm:degree} and  \ref{thm:degreeeinstein} respectively Theorems \ref{thm:hprotrafo} and  \ref{thm:hprotrafoeinstein} is that in the general case we divide ${n+1-k}$ by $2$ and in the Einstein case we divide ${n+1-k}$ by $3$. The additional difference is that $k$ goes up to $n-1$ in the general case and up to $n-2$ in the Einstein case. 
As a by-product we also obtain that   in dimension $2n=4$, two c-projectively equivalent Riemannian  Einstein metrics that are not affinely equivalent must be of constant holomorphic  curvature;  this result was known before and is in \cite{haddad1}. 

Under the assumption that the manifold is closed, the list of the degrees of mobility  is  again much more simple as the next theorem shows:

\begin{thm}\label{thm:einstein}
Suppose $g$ and $\tilde{g}$ are c-projectively equivalent K\"ahler-Einstein metrics of arbitrary signature on a closed connected complex manifold $(M,J)$ of real dimension $2n\geq 4$.
Then, $g$ and $\tilde{g}$ are affinely equivalent unless  $(M,g,J)$ is $(\mathbb{C}P(n),\mbox{const}\cdot g_{FS},J_{standard})$.
\end{thm}

Note that as examples constructed in \cite{ApostolovII} show, the assumption that the second metric is also Einstein is important for Theorem \ref{thm:einstein}.

As a by-product, in the proof of Theorems  \ref{thm:hprotrafoeinstein} and  \ref{thm:einstein} we obtain the following

\begin{thm} \label{+2} {Assume two nonproportional  Kähler-Einstein metrics (of arbitrary signature) on a complex $2n\ge 4$-dimensional  manifold $(M,J)$ 
are c-projectively equivalent. Then, any  Kähler metric that is c-projectively equivalent to them is also Einstein.} \end{thm}

Modulo Theorem \ref{+2}   and under the additional assumption that $g$ is Riemannian,  Theorem \ref{thm:einstein} follows from known global results in Kähler-Einstein geometry\footnote{We are grateful to D. Calderbank and C. T\o nnesen-Friedman for pointing this out to us.}. 
More precisely,  if the first Chern class $c_1(M)$ is nonpositive, Theorem \ref{thm:einstein} follows from  \cite[Theorem 1]{Sinjukov2}. The proof of  \cite{Sinjukov2} is a  standard application of the  Weitzenböck formula which shows nonexistence of a Hamiltonian Killing vector field, see also  \cite[page 89]{besse} or \cite[page 77]{moroianu} (it is known that   nonaffine c-projective equivalence implies the existence of a  Hamiltonian Killing vector field, see e.g. \cite{ApostolovI,Kiyohara2010}). 

If $c_1(M)>0$, then    by Bando and Mabuchi   \cite{bandomabuchi} (see also  \cite[Addendum D]{besse}), for any two K\"ahler-Einstein metrics $g,\tilde{g}$ on a closed connected complex manifold $(M,J)$,
 there exists a bi-holomorphism $f:M\rightarrow M$, contained in the connected component of the group of bi-holomorphic transformations of $(M,J)$, such that $f^* g=\textrm{const} \cdot \tilde{g}$. This bi-holomorphism is of course a c-projective transformation of $g$. 
  Now, by Theorem \ref{+2}, we have a one-parameter family of c-projectively equivalent  Einstein metrics which gives us a one-parameter family of such bi-holomorphisms.  
  By the standard rigidity argument, we obtain then  the existence a nontrivial (i.e., containing not only isometries) 
    connected Lie group of c-projective  transformations and  \cite[Theorem 1]{MatRos} implies that the manifold is $(\mathbb{C}P(n),\mbox{const}\cdot g_{FS},J_{standard})$.

As a direct corollary of Theorem \ref{thm:einstein} we obtain  
\begin{cor}[Generalization of the Yano-Obata conjecture to metrics of arbitrary siganture under the additional assumption that the metrics are Einstein]
Let $(M,g,J)$ be a closed Kähler-Einstein  manifold of arbitrary signature of real dimension $2n\geq 4$. Then, every c-projective vector field is an affine (i.e., connection-preserving) vector field unless $(M,g,J)=(\mathbb{C}P(n),\mbox{const}\cdot g_{FS},J_{standard})$.    
\end{cor}

\subsection{Relation to hamiltonian and conformal Killing $2$-forms}\label{sec:reltoothers}

Let $(M,g,J)$ be a K\"ahler manifold of real dimension $2n$ with K\"ahler form $\omega=g(.,J.)$. A hermitian $2$-form $\phi$ is called \emph{hamiltonian $2$-form} if
\begin{align}
\nabla_X \phi=X^\flat\wedge J\lambda+(JX)^\flat\wedge\lambda\label{eq:hamiltonian2form}
\end{align}
for a certain $\lambda\in \Omega^1(M)$, where we define $J\lambda=\lambda\circ J$. It is straight-forward to see that $\lambda=\tfrac{1}{4}d\,\mathrm{trace}_\omega \phi$, where $\mathrm{trace}_\omega \phi=\sum_{i=1}^{2n}\phi(Je_i,e_i)$ for an orthonormal frame $e_1,...,e_{2n}$. Hamiltonian $2$-forms where introduced in \cite{Apostolov0} and have been studied further in \cite{ApostolovI,ApostolovII,ApostolovIII,ApostolovIV}. Among other interesting results and applications in K\"ahler geometry \cite{ApostolovIII,ApostolovIV}, a complete local \cite{ApostolovI} and global \cite{ApostolovII} classification of Riemannian K\"ahler manifolds admitting hamiltonian $2$-forms have been obtained.

Let $\phi$ be a hamiltonian $2$-form and consider the corresponding symmetric hermitian $(0,2)$-tensor $A=\phi(J.,.)$. It is easy to see that this correspondence sends solutions of \eqref{eq:hamiltonian2form} to solutions of \eqref{eq:mainA} and vice versa, thus hamiltonian $2$-forms and hermitian symmetric solutions of \eqref{eq:mainA}, i.e. c-projectively equivalent K\"ahler metrics, are essentially the same objects. We immediately obtain from Theorems \ref{thm:degree} and \ref{thm:degreeeinstein} the next two corollaries.
\begin{cor}\label{cor:hamilt}
Let $(M,g,J)$ be a simply connected Riemannian K\"ahler manifold of real dimension $2n\geq 4$. Suppose there exist at least one hamiltonian $2$-form on $M$ that is not parallel. 

Then, the dimension of the space of hamiltonian $2$-forms belongs to the following list: 
\begin{itemize}
\item $2$,
\item $k^2+\ell$, where $k=0,...,n-1$ and $\ell=1,...,\big[\tfrac{n+1-k}{2}\big]$.
\item $(n+1)^2$,
\end{itemize}

Moreover, every  value  from  this  list that is greater than or equal to $2$ is the dimension of the space of hamiltonian $2$-forms of a certain simply connected  $2n$-dimensional K\"ahler manifold $(M, g,J)$ that admits a non-parallel hamiltonian $2$-form.
\end{cor}

\begin{cor}\label{cor:hamilteinstein}
Let $(M,g,J)$ be a simply connected Riemannian K\"ahler-Einstein manifold of real dimension $2n\geq 4$. Assume that there exists at least one  
hamiltonian $2$-form on $M$  that is not parallel. Then, the dimension of the space of hamiltonian $2$-forms is equal to one of the following numbers:
\begin{itemize}
\item $2$,
\item $k^2+\ell$, where $k=0,...,n-2$ and $\ell=1,...,\big[\tfrac{n+1-k}{3}\big]$,
\item $(n+1)^2$. 
\end{itemize}

Conversely, each of the numbers of the above list that is greater than or equal to $2$ is the dimension of the space of hamiltonian $2$-forms of a certain 
simply connected  Riemannian  K\"ahler-Einstein manifold $(M^{2n},g,J)$ that admits a nonparallel  hamiltonian $2$-form.
\end{cor}

Following \cite[Appendix A]{ApostolovI} and \cite{moruianusemmelmann}, we briefly recall the relation between hamiltonian $2$-forms and conformal Killing or twistor $2$-forms studied in \cite{moruianusemmelmann,semmelmann}. On a $m$-dimensional Riemannian manifold $(M,g)$ a \emph{conformal Killing or twistor $2$-form} is a $2$-form $\psi$ satisfying

\begin{align}
\nabla_X \psi=X^\flat\wedge \alpha+i_X \beta\label{eq:defkilling2form}
\end{align}
for certain $\alpha\in \Omega^1(M)$ and $\beta\in \Omega^3(M)$. It is straight-forward to see that 
$$\beta=\frac{1}{3}d\psi\mbox{ and }\alpha=-\frac{1}{n-1}\delta \psi,$$
where $(\delta \psi)(X)=\sum_{i=1}^m (\nabla_{e_i}\psi)(X,e_i)$ for an orthonormal frame $e_1,...,e_m$.

Now assume that $(M,g,J)$ is a K\"ahler manifold of real dimension $m=2n$, with K\"ahler form $\omega=g(.,J.)$. Then, as shown in \cite[Appendix A]{ApostolovI} and \cite[Lemma 3.11]{moruianusemmelmann}, the defining equation \eqref{eq:defkilling2form} for a conformal Killing $2$-form which in addition is assumed to be hermitian, equivalently reads
\begin{align}
\nabla_X \psi=X^\flat\wedge \alpha-(JX)^\flat\wedge J\alpha+J\alpha(X)\omega.\label{eq:defkilling2formherm}
\end{align}

On the other hand, setting $\alpha'=J\lambda$, \eqref{eq:hamiltonian2form} becomes
\begin{align}
\nabla_X \phi=X^\flat\wedge \alpha'-(JX)^\flat\wedge J\alpha'\label{eq:hamiltonian2form2}
\end{align}
Thus, if $\phi$ is a hamiltonian $2$-form we set $f_\lambda=\tfrac{1}{4}\mathrm{trace}_\omega \phi$ and see that 
\begin{align}
\psi=\phi-f_\lambda\omega\label{eq:hamiltoconf}
\end{align}
is a conformal Killing $2$-form with $\alpha=\alpha'$. Conversely, let $\psi$ be a conformal Killing $2$-form. In the case $2n>4$, we have $J\alpha=df_\alpha$, where $f_\alpha=\tfrac{1}{(2n-4)}\mathrm{trace}_\omega\psi$, see \cite[Appendix A]{ApostolovI} and \cite[Lemma 3.8]{moruianusemmelmann}. Then,
\begin{align}
\phi=\psi-f_\alpha\omega\label{eq:conftohamil}
\end{align}
is a hamiltonian $2$-form with $\alpha'=\alpha$, and the linear mappings in \eqref{eq:hamiltoconf} and \eqref{eq:conftohamil} are inverse to each other. Thus, when $2n>4$, hamiltonian $2$-forms and conformal (hermitian)  Killing $2$-forms are essentially the same objects and Corollaries \ref{cor:hamilt} and \ref{cor:hamilteinstein} also descibe the possible dimensions of the space of  hermitian conformal Killing $2$-forms  for simply-connected Riemannian K\"ahler respectively Riemannian K\"ahler-Einstein manifolds of dimension $2n>4$ that  admit a non-parallel conformal Killing $2$-form.

\subsection{Relation to  parallel $(0,2)$-tensors on the conification and circles of ideas used in  the proof}
\label{sec:ideas}

As we mentioned above and will recall later, the system  of equations  \eqref{eq:mainA} (such that the dimension of the space of solutions  is the degree of mobility) is linear and overdetermined. In theory, there exist algorithmic (sometimes called prolongation-projection or Cartan-Kähler) methods  to  understand the dimension of the space of solutions of such a system. In practice, these methods are effective in small dimensions only, or return the maximal and submaximal values for the possible dimension of the space of solutions. Actually,  the previously known results are obtained by  versions of  these methods.

A key observation that allowed us to solve  the problem in its full generality is that the problem can be  reduced  to the following geometric one:

{\it Find all possible dimensions of the spaces of parallel symmetric hermitian (0,2)-tensors on conifications of Kähler manifolds.}

The reduction goes as follows. First of all, at least for the proof of Theorem \ref{thm:degree}, we may assume that the degree of mobility of $(M^{2n},g,J)$ is $\ge 3$ (since otherwise it is $2$ and this case is clear).   
 
Then, as we recall in Theorem \ref{thm:hprosystem}, the solutions of the system  \eqref{eq:mainA} are in one-one correspondence with the solutions of the system \eqref{eq:hprosystem}. 

Since the construction of the conification $(\hat M^{2n+2}, \hat g, \hat J)$ of $(M^{2n},g,J)$ in general only works locally (we explain this in detail in Section \ref{sec:construction} and Theorem \ref{thm:coneconstruction1}) we apply it assuming  that our manifold is diffeomorphic to a disc. The extension of the results to simply connected manifolds is then a standard application of the theorem of Ambrose-Singer and will be explained in Section \ref{sec:proofthmdegreeglobal}. 

Consider now the number $B$ from \eqref{eq:hprosystem}. If $B=-1$, then the solutions of  \eqref{eq:hprosystem} are essentially parallel symmetric hermitian $(0,2)$-tensors on $(\hat M, \hat g, \hat J)$, see  Theorem \ref{thm:coneconstruction2}. If $B\ne -1$ but $B\ne 0$, one can always make $B=-1$ by an appropriate scaling of the metric $g$. If the initial constant $B$ was positive, then $\hat{g}$ has signature $(2,2n)$.  

The case $B=0$ requires the following additional work: we show that on any simply connected neighborhood such that its closure is compact there exists a c-projectively equivalent Riemannian metric  with $B\ne 0$, see  Section \ref{sec:mainfeatures}. For this new metric the above reduction works. Evidently, c-projectively equivalent metrics have the same  degrees of mobility. Thus, also in this case, instead of solving the initial problem, we  study the possible  dimensions of the space of parallel symmetric hermitian $(0,2)$-tensors on the conification of a Kähler manifold. 
 
To do this, let us first assume, for simplicity and since the ideas will be already visible in this setting, that $B$ is negative so the conification we will work with is a Riemannian manifold. In the final proof we will assume that the signature is $(2,2n)$ which poses additional difficulties. In order to  calculate the dimension of the space of parallel symmetric hermitian $(0,2)$-tensors on the conification $(\hat M, \hat g, \hat J)$, let us consider the (maximal orthogonal) holonomy decomposition of the tangent space  at a certain point of $\hat M$:
\begin{align}
T_p\hat M=T_0\oplus T_1\oplus...\oplus T_\ell\label{eq:decomptangenttrivial-1}.
\end{align}
where $T_0$ is flat (in the sense that the holonomy group acts trivially on it) and all other   $T_i$    are irreducible. Clearly, each $T_i$ is $\hat J$-invariant  so it has even dimension $2k_i$. 

It is well known, at least since de Rham \cite{DR}, that   parallel  symmetric $(0,2)$-tensor fields  on a Riemannian
  manifold $\hat M$ are in one-one correspondence with the $(0,2)$-tensors on $T_p \hat M$ of the form 
\begin{equation}\label{presentation}
\sum_{i,j=1}^{2k_0}c_{ij} \tau_i  \otimes  \tau_j + C_1 g_1 + ...+ C_\ell  g_\ell,  
\end{equation}
where $\{\tau_i\}$ is a basis in $T_0^*$ and $g_i$ is  the restriction of $\hat g$ to $T_i$.  The assumption that the tensor is symmetric and hermitian implies that the $2k_0\times 2k_0$-matrix $c_{ij}$ is symmetric and is  hermitian w.r.t. to the restriction of $\hat J$ to $T_0$. This gives us a $k_0^2 $-dimensional space of such tensors. If  $\hat M$   is flat, i.e.  if  the initial metric $g$ has constant holomorphic  curvature, we  have $k_0= n+1$ (=half of the dimension of the conification of our $2n$-dimensional $M$) which gives us the number $(n+1)^2$. Suppose now that our manifold is not flat so $\ell\ge 1$. We show that the dimension of each $T_i$ for $i\ge 1$  is $\ge 4$ (i.e., that $k_i\ge 2$). The key observation that is used here is that each $T_i$ (or, more precisely, the restriction of $\hat g$ to the integral leafs of the integrable distribution $T_i$) is a cone manifold. Since it is not flat, it  has dimension $>2$ and since the dimension is even it must be $\ge 4$.
 Then, $k=0,...,n-1$ and $\ell$ is at most $ \big[\tfrac{n+1-k}{2}\big]$, because the sum of the dimensions of  all $T_i$ with  $i\ge 1$ is $2(n+1)- 2k_0$ and each $T_i$ ``takes'' at least four  dimensions, see \eqref{+1}.  
            
Now, in the case our manifold is Einstein, we show that the dimension of  each $T_i$ is $\ge 6$. This statement is due to the nonexistence of Ricci-flat but nonflat cones of dimension $4$. Then, $k=0,...,n-2$ and $\ell$ is at most $ \big[\tfrac{n+1-k}{3}\big]$,  which gives us the list from Theorem  \ref{thm:degreeeinstein}.
            
In the case the conification of the manifold has signature $(2,2n)$, we use recent results of \cite[Theorem 5]{FedMat}, where  an analog of \eqref{eq:decomptangenttrivial-1}  was  proven under the additional assumption that our manifold of signature $(2,2n)$ is a cone manifold (and conifications of Kähler manifolds  of dimension $2n$ are indeed cone manifolds over $2n+1$-dimensional manifolds). The additional work is required though since  the number $k$ in this case is not  necessary $\tfrac{\dim(T_0)}{2}$   
but  could also  be $\tfrac{\dim(T_0)}{2} + 1$. In the latter case we show that one of $T_i$ with $i\ge 1$ must necessary have dimension $\ge 6$ (resp. $\ge 8$ in the Einstein  case) and  the list of  possible dimensions of the space of parallel symmetric hermitian $(0,2)$-tensors remains the same.

Let us now touch  the proof of   Theorems \ref{thm:hprotrafo} and \ref{thm:hprotrafoeinstein}. The restriction $\mathrm{dim}(\mathfrak{c}(g,J)/\mathfrak{i}(g,J))\le D(g,J)-1$ is straightforward.  Now, in the case $D(g,J)\geq 3$, under the additional assumption  that the constant $B\ne 0$, we  actually have 
\begin{equation} \label{-2} \mathrm{dim}(\mathfrak{c}(g,J)/\mathfrak{i}(g,J))= D(g,J)-1.\end{equation}  

Indeed, for every solution of \eqref{eq:mainA}  we canonically construct  a c-projective vector field. Two solutions of \eqref{eq:mainA} give the same c-projective vector field if and only if their  difference is 
a multiple of $g$.  The formula  \eqref{-2} gives us the lists from  Theorems \ref{thm:hprotrafo} and \ref{thm:hprotrafoeinstein}. The case when $B=0$ can be reduced to the case  $B\ne 0$ by the same trick as in the proofs of  Theorems \ref{thm:degree} and \ref{thm:degreeeinstein}.

\subsection{Organization of the paper}

In Section \ref{sec:basics} we recall basic statements  in c-projective geometry that were proved before and that will be used in the proof. 

In Section \ref{sec:construction}, we  describe the construction of the conification $(\hat{M},\hat{g},\hat{J})$ of a K\"ahler manifold $(M,g,J)$ and will show that solutions to the system \eqref{eq:hprosystem} with $B=-1$ on $M$ correspond to parallel hermitian symmetric $(0,2)$-tensors on $\hat{M}$.

Section \ref{sec:mainfeatures} is a technical one, its first goal is to explain that w.l.o.g. we can assume that $B$ in \eqref{eq:hprosystem}  is equal to $-1$. The only nontrivial step here is in Section  \ref{sec:proofoflemma} and is  as follows:  if for the initial metric $B=0$, we change  the metric to a c-projectively equivalent one such that $B\neq 0$ for the new metric.  The second goal of Section \ref{sec:mainfeatures}  is to prove the additional statement concerning the case of an Einstein metric. 
Roughly speaking, one  goal is to show that  when we change the metric to  make $B\ne 0$, the metric we obtain is still  Einstein.  Another goal is to show that 
the conification construction applied  to K\"ahler-Einstein metrics gives a Ricci-flat metric.

In Section \ref{sec:proofs} we prove the Theorems \ref{thm:degree} and \ref{thm:degreeeinstein}. In Section \ref{sec:proofthmdegree} we will essentially calculate the possible dimensions of the space of parallel symmetric hermitian $(0,2)$-tensors on K\"ahler manifolds that arise from the conification construction. This will complete the proofs of the Theorems \ref{thm:degree} and \ref{thm:degreeeinstein} in the local situation. In Section \ref{sec:proofthmdegreeglobal} we will extend our local results to the global situation. 
In Section \ref{sec:realization} we complete the proof of Theorem \ref{thm:degree} respectively Theorem \ref{thm:degreeeinstein} and show that each of the values in these theorems is  the degree of mobility of a certain K\"ahler metric respectively K\"ahler-Einstein metric.

In Section \ref{sec:thmhprotrafo} we   prove  Theorems \ref{thm:hprotrafo} and \ref{thm:hprotrafoeinstein} and in  the final Section \ref{sec:thmeinstein}  we prove  Theorem \ref{thm:einstein}.

\section{Basic facts in the theory of c-projectively equivalent metrics}\label{sec:basics}

\subsection{C-projective equivalence of K\"ahler metrics as a system of PDE}
Let $g$ and $\tilde{g}$ be two K\"ahler metrics on the complex manifold $(M,J)$ of real dimension $2n\geq 4$. 

It is well-known, see \cite[equation (1.7)]{Tashiro1956}, that $g$ and $\tilde{g}$ are c-projectively equivalent if and only if for a certain $1$-form $\Phi$, the Levi-Civita connections $\nabla,\tilde{\nabla}$ of $g,\tilde{g}$ respectively satisfy
\begin{align}
\tilde{\nabla}_X Y-\nabla_X Y=\Phi(X)Y+\Phi(Y)X-\Phi(JX)JY-\Phi(JY)JX\label{eq:lchpro}
\end{align}
for all vector fields $X,Y$. In the tensor index notation, \eqref{eq:lchpro} reads
 \begin{equation}\label{M1}
 \tilde \Gamma^i_{jk} - \Gamma^i_{jk}=  \delta^i_j \Phi_k +  \delta^i_k \Phi_j - J^i_{\ j} \Phi_sJ^s_{\ k}  -  J^i_{\ k} \Phi_s J^s_{\ j}.
 \end{equation}
 
 Actually, the $1$-form $\Phi$  is  exact,  i.e., it is the differential of a function, and the function is explicitly given in terms of $g$ and $\tilde g$. Indeed, contracting \eqref{M1} w.r.t. $i$ and $k$, we obtain 
 $$
 \tfrac{1}{2}\partial_i\ln \left(\frac{\det \tilde g}{ \det g}\right)= 2(n+1) \Phi_i  
 $$
 so $\Phi_i= \phi_{,i} = d\phi$ for the function $\phi$ given by 
 \begin{equation}\label{M2}
 \phi:= \tfrac{1}{4(n+1)}\ln \left(\frac{\det  \tilde  g}{\det g}\right). 
 \end{equation}
 
The equation \eqref{M1} allows us to reformulate the condition that the metrics $g$ and $\tilde g$ are c-projectively equivalent as a system of PDE on the components of $\tilde g$ whose coefficients depend on $g$ (and its derivatives). 
Indeed, in view of \eqref{eq:lchpro}, the condition $\tilde{\nabla} \tilde{g}=0$ which is the defining equation for $\tilde \nabla$ 
reads 
\begin{equation}\label{eisenhart}
\nabla_Z \tilde{g}(X,Y) -  \Phi(X) \tilde{g}(Z,Y)-   \Phi(Y) \tilde{g}(Z,X) -  \Phi(JX) \tilde{\omega}(Z,Y) - \Phi(JY) \tilde{\omega}(Z,X)- 2 \Phi(Z) \tilde{g}(X,Y)=0, 
\end{equation}
where we denote by $\tilde{\omega}$ the K\"ahler $2$-form $\tilde{\omega}=\tilde{g}(.,J.)$

We will view this  condition as a system of PDE on the entries of $\tilde g$ whose coefficients depend on the entries of $g$. This system of equations is nonlinear (since  the entries of $\Phi$  depend algebraically on the entries $\tilde g$). A  remarkable observation by Mikes et al  \cite{DomMik1978} is that one  can make this system   linear 
by a clever  substitution. For this,  consider the  symmetric hermitian $(0,2)$-tensor $A$ and the $1$-form $\lambda$ given by  
\begin{eqnarray}
A&=&A(g,\tilde{g})=\left(\frac{\mathrm{det}\,\tilde{g}}{\mathrm{det}\,g}\right)^{\tfrac{1}{2(n+1)}}g\tilde{g}^{-1}g.\label{eq:defA} \\ 
\lambda& =& -\Phi g^{-1} A.  \label{eq:relationphilambda}
\end{eqnarray}
Here $g^{-1}$  is viewed as an isomomorphism $g^{-1}:T^*M\to TM$ given by the condition  $g(g^{-1}\psi, Y):=\psi(Y) $ for all $Y$ 
 and $A$ is viewed as a mapping $A:TM\to T^*M$ given by the condition $AX(Y)= A(X,Y)$. If we view $(0,2)$-tensors as matrices and  $(0,1)$-tensors as $2n$-tuples,
  the matrix of $A$ (up to multiplying it with the scalar expression involving the determinants of the metrics) is the product of the matrices  $g$,  $\tilde g^{-1}$, and again $g$. The 1-form   $\lambda$ in the matrix-notation 
   is minus the product  $\Phi$,  $g^{-1}$ and $A$. 
 Using index notation, 
 $$A_{ij} = \left(\frac{\mathrm{det}\,\tilde{g}}{\mathrm{det}\,g}\right)^{\tfrac{1}{2(n+1)}}  g_{is} \tilde g^{sr} g_{rj} \textrm{ and    $\lambda_i = -\Phi_s g^{sr} A_{r i}$ } ,$$
 where $\tilde g^{ij}$ is the inverse $(2,0)$-tensor to $\tilde g_{ij}$.
 
Straightforward  calculations show that the condition \eqref{eisenhart}  is  equivalent to the formula \eqref{eq:mainA}.  In index notation, the equation \eqref{eq:mainA} reads
  \begin{equation}\label{Mbasic}
  a_{ij,k} = \lambda_i g_{jk}+  \lambda_j g_{ik}+   J^s_{ \ j}\lambda_s J_{ik}+  J^s_{ \ i}\lambda_s J_{jk}. 
  \end{equation}
Note that contracting the equation  \eqref{Mbasic} with $g^{ij}$ we obtain that the one-form  $\lambda$ is the differential of the function   $\tfrac{1}{4}\mathrm{trace}_g(A)$.   In view of this, \eqref{eq:mainA} could be viewed as a linear  system 
of PDE on the entries of $A$ only. 

Note also that the formula \eqref{eq:defA} is invertible: the tensor $A$ given by  \eqref{eq:defA} is nondegenerate, and the metric $\tilde g$ is given in the terms of $g$ and $A$ by 
\begin{align}
\tilde g=\tfrac{1}{\sqrt{\det(A)}} gA^{-1}g.\label{eq:inverse}
\end{align}

\begin{rem}\label{rem:affine}
The metrics $g,\tilde{g}$ are affinely equivalent, if and only if the tensor $A$ in \eqref{eq:defA} is parallel, i.e., if and only of $\lambda$ in \eqref{eq:mainA} is identically zero.
\end{rem}

\begin{rem}\label{rem:metricsol}
Since \eqref{eq:mainA} is linear and the metric $g$ is always a solution of \eqref{eq:mainA}, for every solution $A$ of \eqref{eq:mainA},  the tensor $A+\mathrm{const}\cdot g$ is again a solution. Thus, if $A$ is degenerate, we can choose (at least locally)  the constant such that $A+\mathrm{const}\cdot g$ is a non-degenerate solution and, hence, corresponds to a metric $\tilde{g}$ that is  c-projectively equivalent to $g$. 
\end{rem}

Let us denote by $\mathcal{A}(g,J)$ the linear space of symmetric hermitian solutions $A$ of \eqref{eq:mainA}.  The \emph{degree of mobility $D(g,J)$} of a K\"ahler structure $(g,J)$ is the dimension of $\mathcal{A}(g,J)$.

\begin{rem}\label{rem:isomorphism}
If the metrics $g,\tilde{g}$ are c-projectively equivalent, the spaces $\mathcal{A}(g,J)$ and $\mathcal{A}(\tilde{g},J)$ are isomorphic and hence, $D(g,J)=D(\tilde{g},J)$. This statement is probably expected and evident, the proof  can be found for example in \cite[Lemma 1]{MatRos}. 
\end{rem}

The  results recalled above are rather classical. The next result \cite[Theorem 3]{FKMR} is a  recent one  and 
plays  a  key role  in our paper.

\begin{thm}[\cite{FKMR}]\label{thm:hprosystem}
Let $(M,g,J)$ be a connected K\"ahler manifold of degree of mobility $D(g,J)\geq 3$ and of real dimension $2n\geq 4$. Then,   there exists a unique constant $B$ such that 
for every $A\in \mathcal{A}(g,J)$  with the corresponding $1$-form $\lambda$ there exists  the unique   function $\mu$  such that  $(A,\lambda,\mu)$ satisfies
\begin{align}
\begin{array}{c}
(\nabla_Z) A (X,Y)=g(Z,X)\lambda(Y)+g(Z,Y)\lambda(X)+\omega(Z,X)\lambda(JY)+\omega(Z,Y)\lambda(JX)\vspace{1mm}\\
(\nabla_Z \lambda) (X)=\mu g(Z,X)+BA(Z,X)\vspace{1mm}\\
\nabla_Z \mu =2B\lambda(Z).
\end{array}\label{eq:hprosystem}
\end{align}
for all vector fields $X,Y,Z$. 
\end{thm}

\section{Conification construction and solutions of \eqref{eq:hprosystem} as  parallel tensors on the conification}
\label{sec:construction}

Let $(M,g,J)$ be a K\"ahler manifold of arbitrary signature with K\"ahler $2$-form $\omega=g(.,J.)$. We explicitly allow $\dim M = 2n= 2$ here.  We also  suppose that  the form $\omega$ is an exact form, $\omega=d\tau$. This is always true if  $H^2(M,\mathbb{R})=0$ and in particular, it is always true if the manifold is diffeomorphic to the ball. This is sufficient for our purposes since we will apply the conification construction only to subsets of such kind.  
 
We consider the manifold $P=\mathbb{R}\times M,$ where $t$ will denote the standard coordinate on $\mathbb{R}$ and the natural projection to $M$ is denoted by $\pi:P \to M$. 
  
On $P$, we define the $1$-form 
$$\theta=dt-2\tau,$$
where for readability we omit the  symbol for the pullback of $\tau$ to $P$ (actually, the formula above should be $\theta= dt - 2 \pi^*\tau$). We will also omit the symbols of the $\pi$-pullback in all formulas below so if in some formula we sum or compare a $(0,k)$-tensor defined on $M$ with  a $(0,k)$-tensor  defined on $P$,  the tensor on $M$ should be pulled back to $P$ by $\pi$. 

Clearly,  $d\theta=-2d\tau=-2\omega$. Let us define the metric $h$ on $P$  by
$$h=\theta^2+g \  \ (\textrm{where $\theta^2 =\theta\otimes\theta$}).$$ 
\begin{rem}\label{rem:changeoftau}
The freedom in the choice of $\tau$ is not essential for us since (as it is  straightforward to check)  the change $\tau\longmapsto \tilde{\tau}=\tau+df$ yields a metric $\tilde{h}$ that is isometric to $h$: the mapping  $\phi(t,p)=(t-2f(p),p)$ satisfies $\phi^*\theta=\tilde{\theta}$ and hence, $\phi^* h=\tilde{h}$.  
\end{rem}
Further, let us denote by $\mathcal{H}=\mathrm{kern}(\theta)$ the ``horizontal'' distribution of $P$ defined by $\theta$. 
For any vector field $X$ on $M$ we define its \emph{horizontal lift} $X^\theta$ on $P$ by the properties 
\begin{align}
   \pi_{*}(X^\theta)= X \textrm{ and } \theta(X^\theta)=0.\label{eq:horizontallift}
\end{align}

We consider now  the \emph{cone} over $(P,h)$, that is  we consider the $2(n+1)$-dimensional  manifold $(\hat{M}=\mathbb{R}_{>0}\times P,\hat{g}=dr^2+r^2h)$. Let us denote by $\hat \pi:\hat M\to M $ the natural projection $(r,t,p)\mapsto p \in M$. The kernel of the differential of this projection is spanned by  
  $\xi=r\partial_r$, which will be called the \emph{cone vector field}, and $\eta=\partial_t$. In the literature on Sasaki manifolds, $\eta$ is sometimes refered to as Reeb vector field. 
  
Next, we  introduce an  almost complex structure $\hat{J}$  on $\hat M$ (which later appears to be a complex structure) by the formula  
$$\hat{J}\xi=\eta, \hat{J}\eta=-\xi,\   \mbox{ and } \hat{J}X^{\theta}=(JX)^{\theta}. $$

We will call the $2(n+1)$-dimensional manifold $(\hat{M},\hat{g},\hat{J})$  the \emph{conification} of $(M,g,J)$.

We took the name ``conification'' from the recent paper \cite{ACM} of Alekseevsky et al where this construction has been obtained in a more general situation. The relation between the two constructions is explained in \cite[Example 1]{ACM}.       

We have
\begin{thm}[essentially, \cite{ACM}]\label{thm:coneconstruction1}
Let $(M,g,J)$ be a K\"ahler manifold of real dimension $2n\geq 2$ and suppose that the K\"ahler $2$-form $\omega=g(.,J.)$ satisfies $\omega=d\tau$ for a certain $1$-from $\tau$. Then the conification $(\hat{M},\hat{g},\hat{J})$ is a K\"ahler manifold. 

Conversely, suppose $(\hat{M},\hat{g},\hat{J})$ is a K\"ahler manifold which locally, in a neighborhood of every point of a dense open subset, is a cone, i.e. $(\hat{M},\hat{g})$ is of the form $(\hat{M}=\mathbb{R}_{>0}\times P,\hat{g}=dr^2+r^2h)$ for a certain pseudo-Riemannian manifold $(P,h)$. Then, $(\hat{M},\hat{g},\hat{J})$ arises locally as the conification of its K\"ahler quotient $(M,g,J)$.
\end{thm}
The K\"ahler quotient is taken with respect to the action of the hamiltonian Killing vector field $\eta=\hat{J}\xi$, where $\xi=r\partial_r$ is the cone vector field. This will be explained in more detail in the proof of Theorem \ref{thm:coneconstruction1}. 

We could also assume the less restrictive condition that the K\"ahler class $[\omega]\in H^{2}(M,\mathbb{R})$ is integer, and hence, $\omega$, up to scale, is the curvature of a certain connection one-form $\theta$ on some $S^1$-bundle $P$ over $M$ as in \cite{ACM}. Then, the metrics $h$ on $P$ and $\hat{g}$ on the cone $\hat{M}$ over $P$ can be defined in the same way as above and the proof of Theorem \ref{thm:coneconstruction1} will be literally the same as the proof that we will give below. In this more general situation, the role of $\eta=\partial_t$ is then played by the fundamental vector field $\eta$ of the $S^1$-action on $P$ that satisfies $\theta(\eta)=1$.

Actually from the construction it is immediately clear that $\hat J^2 = -\Id$. It is also straight-forward to check that $\hat g$ is hermitian w.r.t. $\hat J$ by checking the condition $\hat{g}(\hat{J}u,\hat{J}v)=\hat{g}(u,v)$ for all possible combinations of tangent vectors $u,v$ of the form $\xi,\eta$ and $X^\theta$. What remains is to show that $\hat J$ is parallel with respect to the Levi-Civita connection of $\hat g$. This will be done in Section \ref{sec:proof of coneconstruction1}. 

Let us explain how the system \eqref{eq:hprosystem} on $M$ relates to the parallel $(0,2)$-tensors on $\hat{M}$.
\begin{thm}\label{thm:coneconstruction2}
Let $(M,g,J)$ be a K\"ahler manifold of real dimension $2n\geq 4$,  suppose that the K\"ahler $2$-form $\omega=g(.,J.)$ satisfies $\omega=d\tau$ for a certain $1$-from $\tau$.  Then, there exists an isomorphism between the space of solutions $(A,\lambda,\mu)$ of \eqref{eq:hprosystem} with $B=-1$ and the space of parallel symmetric hermitian $(0,2)$-tensors $\hat{A}$ on $(\hat{M},\hat{g},\hat{J})$.  The isomorphism is explicit and  is given by
  \begin{align}
(A,\lambda,\mu)\leftrightarrow \hat{A}=\mu dr^2-r dr\odot\lambda+r^2(\mu\theta^2+\theta\odot\lambda(J.)+A),\label{eq:Acccone}
\end{align}
where we omit the symbol for the $\hat \pi$-pullback of $\mu,\lambda,\lambda(J.)$ and $A$ to $\hat{M}$. In the formula \eqref{eq:Acccone}, $X\odot Y=X\otimes Y+Y\otimes X$ is the symmetric tensor product. 
\end{thm}
\begin{rem}
If we use $\tilde{\tau}=\tau+df$ instead of $\tau$ to construct the conification, the diffeomorphism $\phi:\hat{M}\rightarrow \hat{M}$ given by $\phi(r,t,p)=(r,t-2f(p),p)$ (compare also Remark \ref{rem:changeoftau}) satisfies $\phi^* \theta=\tilde{\theta}$. Thus, $\phi$ sends $\hat{g}$, the K\"ahler $2$-form $\hat{\omega}=\hat{g}(.,\hat{J}.)=r\theta\wedge dr+r^2\omega$ and $\hat{A}$ given by \eqref{eq:Acccone} to the corresponding objects constructed by using $\tilde{\tau}$.
\end{rem}

The proof of theorems \ref{thm:coneconstruction1} and \ref{thm:coneconstruction2} is by direct calculations and will be done in  Sections \ref{sec:proof of coneconstruction1} and \ref{sec:isom}. 
The proof of Theorem  \ref{thm:coneconstruction1} is contained in  \cite{ACM} and will be given for self-containedness and because we will need all the formulas from the proof later on.

\subsection{Proof of Theorem \ref{thm:coneconstruction1}.} 
\label{sec:proof of coneconstruction1}

Recall that for a vector field $X$ on $M$, the natural lift of $X$ to the horizontal distribution $\mathcal{H}=\mathrm{kern}(\theta)\subseteq TP$ will be denoted by $X^{\theta}$, see equation \eqref{eq:horizontallift}. 
\begin{lem}
Let $X,Y$ denote vector fields on $M$. The Levi-Civita connection $\hnabla$ of the metric $h$ on $P$ is given by the fomulas
\begin{align}
\hnabla_{\eta}\eta&=0,\,\,\,\hnabla_{\eta}X^{\theta}=\hnabla_{X^{\theta}}\eta=(JX)^{\theta},\,\,\,\hnabla_{X^{\theta}}Y^{\theta}=(\nabla_X Y)^{\theta}+\omega(X,Y)\eta
\label{eq:connsasaki}
\end{align}
\end{lem}
\begin{proof}
Let $u,v,w$ be vector fields on $P$. Using the Koszul formula
$$2h(\hnabla_{u}v,w)=uh(v,w)+vh(w,u)-wh(u,v)-h(u,[v,w])+h(v,[w,u])+h(w,[u,v])$$
we calculate
$$2h(\hnabla_{\eta}\eta,\eta)=0\mbox{ and }2h(\hnabla_{\eta}\eta,X^{\theta})=0,$$
thus, $\hnabla_{\eta}\eta=0$ as we claimed. Now we calculate $2h(\hnabla_{\eta}X^{\theta},\eta)=0$ and 
$$2h(\hnabla_{\eta}X^{\theta},Y^{\theta})=-h(\eta,[X^{\theta},Y^{\theta}])=-\theta([X^{\theta},Y^{\theta}])=d\theta(X^{\theta},Y^{\theta})=-2\omega(X,Y).$$
This shows that $\hnabla_{\eta}X^{\theta}=(JX)^{\theta}$. 

To verify the last equation in \eqref{eq:connsasaki}, we calculate
$$2h(\hnabla_{X^{\theta}}Y^{\theta},\eta)=h(\eta,[X^{\theta},Y^{\theta}])=2\omega(X,Y)\mbox{ and }2h(\hnabla_{X^{\theta}}Y^{\theta},Z^{\theta})=2g(\nabla_X Y,Z).$$
Combining these two equations gives us the third equation in \eqref{eq:connsasaki}.
\end{proof}

The formulas for the Levi-Civita connection $\hat{\nabla}$ of the Riemannian cone $(\hat{M}=\mathbb{R}_{>0}\times P,\hat{g}=dr^2+r^2 h)$ over a pseudo-Riemannian manifold $(P,h)$ are given by
\begin{align}
\hat{\nabla}\xi=\Id,\,\,\,\hat{\nabla}_{X}Y=\nabla_{X}Y-h(X,Y)\xi,\label{eq:LCcone}
\end{align}
where $\xi=r\partial_r$ and $X,Y$ are vector fields on $P$. This is well-known, see for example \cite[Fact 3.2]{Mounoud2010}. For $(P=\mathbb{R}\times M,h=\theta^2+g)$ as above we can combine these formulas with \eqref{eq:connsasaki} to obtain
\begin{align}
\begin{array}{c}
\hat{\nabla}\xi=\Id,\,\,\,\hat{\nabla}\eta=\hat{J},\,\,\,\hat{\nabla}_{\xi}X^\theta=X^\theta,\,\,\,\hat{\nabla}_{\eta}X^\theta=\hat{J}X^\theta,\vspace{1mm}\\
\hat{\nabla}_{X^\theta}Y^\theta=(\nabla_{X}Y)^\theta+\omega(X,Y)\eta-g(X,Y)\xi.
\end{array}\label{eq:LCconncone}
\end{align}
Using these equations, it is easy to check that $\hat{\nabla}\hat{J}=0$. A straight-forward way to do it, is to show that the equation $\hat{\nabla}_u (\hat{J}v)=\hat{J}\hat{\nabla}_u v$ is fulfilled for $u,v$ of the form $\xi,\eta$ and $X^\theta$.

Since $\hat g$ is evidently symmetric, nondegenerate  and hermitian with respect to $\hat J$, the conification $(\hat{M},\hat{g},\hat{J})$ is a K\"ahler manifold as we claimed. 
This completes the proof of  the first statement    of Theorem \ref{thm:coneconstruction1}. 

The other direction of Theorem \ref{thm:coneconstruction1} immediately follows from 
\begin{lem}\label{lem:coneconstruction1b}
Suppose $(\hat{M}^{2n+2},\hat{g},\hat{J})$ is a K\"ahler manifold which is locally, in a neighborhood of every point of a dense open subset, a cone, i.e. $(\hat{M},\hat{g})$ is of the form $(\hat{M}=\mathbb{R}_{>0}\times P,\hat{g}=dr^2+r^2h)$ for a certain $(2n+1)$-dimensional pseudo-Riemannian manifold $(P,h)$. Then, $(\hat{M},\hat{g},\hat{J})$ is  locally the conification of its K\"ahler quotient $(M^{2n},g,J)$, where the quotient is taken w.r.t. the action of the hamiltonian Killing vector field $r\hat{J}\partial_r$.
\end{lem}
\begin{proof}
We work on an open subset of $\hat{M}$ such that on this subset, $(\hat{M},\hat{g})$ is of the form $(\hat{M}=\mathbb{R}_{>0}\times P,\hat{g}=dr^2+r^2h)$. Consider the vector fields  $\xi= r\partial r$ and $\eta=\hat{J}\xi$. Since $\eta $ is orthogonal to $\xi$, the derivative of $r$ in the direction of $\eta$ is zero and therefore $\partial_r$ commutes with $\eta$. Consequently, $\eta$ is essentially a vector field on the manifold $P$, i.e. in a coordinate system $(r,x^1,...,x^{2n+1})$, where $x^1,...,x^{2n+1}$ denote coordinates on $P$, the $\partial_r$-component of $\eta$ is zero and the $\partial_{x^i}$-components do not depend on $r$.

Inserting $\eta$ into the metric $\hat{g}$, we see that $h(\eta,\eta)=1$ and since $\eta$ is a Killing vector field for $\hat{g}$ it follows that $\eta$ is also Killing for $h$. 

Let us take the quotient of $P$ by the action of the local flow of $\eta$, to obtain a (local quotient) bundle $\pi:P\rightarrow M$, where $M$ is a manifold of real dimension $2n$. Since we are working locally anyway, this bundle can be viewed as $P=\mathbb{R}\times M$ with coordinate $t$ on the $\mathbb{R}$-component such that $\eta=\partial_t$. Further, we introduce the $1$-form $\theta=h(\eta,.)$ on $P$ and denote by $\mathcal{H}=\mathrm{kern}\,\theta$, the horizontal distribution. Since $\eta$ is Killing, this distribution is invariant with respect to the action of the flow of $\eta$. For tangent vectors $X,Y\in T_p M$, we denote by $X^\theta,Y^\theta \in \mathcal{H}$ their horizontal lifts to a certain point in $\pi^{-1}(p)\subseteq P$. Defining 
$$g(X,Y)=h(X^\theta,Y^\theta),$$
we see that the right-hand side does not depend on the choice of the base point of $X^\theta,Y^\theta$ in $\pi^{-1}(p)$, hence $g$ defines a Riemannian metric on $M$. 

Consider the endomorphism $J'=\hnabla \eta:TP\rightarrow TP$. From \eqref{eq:LCcone} it immediately follows that 
$$J'\eta=\hnabla_\eta \eta=\hat{\nabla}_\eta \eta-h(\eta,\eta)\xi=\hat{J}\eta+\xi=0.$$
In the same way, we obtain 
$$J'X^\theta=\hnabla_{X^\theta}\eta=\hat{\nabla}_{X^\theta}\eta=\hat{J}X^\theta.$$

Thus, we have that $J':\mathcal{H}\rightarrow \mathcal{H}$ defines an almost complex structure and by setting $(JX)^\theta=J'X^\theta$, we obtain an almost complex structure $J:TM\rightarrow TM$ on $M$ which is indeed independent of the choice of base point of the lift $X^\theta$ of $X\in T_p M$, since the flow of $\eta$ preserves $\hat{J}$.

Using the definition of $(g,J)$ and the fact that $\eta$ is a Killing vector field for $h$, we obtain 
$$g(X,JY)=h(X^\theta,\hnabla_{Y^\theta}\eta)=-\frac{1}{2}d\theta(X^\theta,Y^\theta).$$ 

On the one hand, this shows that $\omega=g(.,J.)$ is a $2$-form or equivalently, that $g$ is hermitian with respect to $J$, i.e., $g(J.,J.)=g$. On the other hand, since $d\theta$ is horizontal in the sense that it vanishes when $\eta$ is inserted, this shows that $d\theta=-2\pi^{*}\omega$. From this, it also follows that $\omega$ is closed. 

Since $\theta([X^\theta,Y^\theta])=-d\theta(X^\theta,Y^\theta)=2\omega(X,Y)$, we obtain that $[X,Y]^\theta=[X^\theta,Y^\theta]-2\omega(X,Y)\eta$. Using this, it is straight-forward to see that the Nijenhuis torsion 
$$N_{J}(X,Y)=[X,Y]-[JX,JY]+J[JX,Y]+J[X,JY]$$
of $J$ lifts to the corresponding Nijenhuis torsion of $\hat{J}$ which is vanishing, more precisely 
$$(N_{J}(X,Y))^\theta=N_{\hat{J}}(X^\theta,Y^\theta)=0.$$
Thus, $J$ is integrable and $(M,g,J)$ is a K\"ahler manifold. From our construction, it is clear that $(\hat{M},\hat{g},\hat{J})$ coincides with the conification of $(M,g,J)$. This completes the proof of the lemma.
\end{proof}
\begin{rem}
The construction of $g$ from $\hat{g}$ as presented in Lemma \ref{lem:coneconstruction1b} is of course well-known and coincides with the K\"ahler quotient of $(\hat{M},\hat{g},\hat{J})$ w.r.t. the action of the hamiltonian Killing vector field $\eta$, see \cite{Hitchin} for a short explanation of K\"ahler quotients and symplectic reduction. The fact that the conification procedure can be reversed by taking the K\"ahler quotient was also mentioned in \cite{ACM}. 
\end{rem}

\subsection{Proof of Theorem \ref{thm:coneconstruction2}}\label{sec:isom}
Let us first recall the following fact proved before for example in \cite[Theorem 8]{FedMat}, \cite[Lemma 1]{Matveev2010} or \cite[Proposition 3.1]{Mounoud2010}: 

\begin{thm}[\cite{FedMat,Matveev2010,Mounoud2010}]
There is an isomorphism between the space of symmetric parallel $(0,2)$-tensors $\hat{A}$ on the Riemannian cone $(\hat{M}=\mathbb{R}_{>0}\times P,\hat{g}=dr^2+r^2h)$ and solutions $(L,\sigma,\rho)\in \Gamma(S^2T^* M\oplus T^*M\oplus \mathbb{R})$ of the linear PDE system
\begin{align}
\begin{array}{c}
(\hnabla_{Z} L) (X,Y)=h(Z,X)\sigma(Y)+h(Z,Y)\sigma(X)\vspace{1mm}\\
(\hnabla_{Z} \sigma) (X)=\rho h(Z,X)-L(Z,X)\vspace{1mm}\\
\hnabla_{Z} \rho =-2\sigma(Z)
\end{array}\label{eq:system}
\end{align}
on $(P,h)$. The isomorphism is explicitly given by 
\begin{align}
(L,\sigma,\rho)\leftrightarrow \hat{A}=\rho dr^2-r dr\odot \sigma+r^2L,\label{eq:isomorphism}
\end{align}
where we omitted the symbol for the pullback of objects from $P$ to $\hat{M}$. 
\end{thm}

Now let $(P=\mathbb{R}\times M,h=\theta^2+g)$ be defined as in the previous section, where $(M,g,J)$ is a $2n$-dimensional K\"ahler manifold with exact K\"ahler $2$-form. Let us prove  a technical lemma that gives us  a characterisation of solutions $(L,\sigma,\rho)$ of the system \eqref{eq:system} on $(P,h)$ which are invariant with respect to the action of the flow of $\eta$:
\begin{lem}\label{lem:dtinvariance}
The solution $(L,\sigma,\rho)$ of \eqref{eq:system} on $(P,h)$ is $\eta$-invariant if and only if
\begin{align}
\sigma(\eta)=0,\,\sigma((JX)^{\theta})=L(\eta,X^{\theta}),L((JX)^{\theta},(JX)^{\theta})=L(X^{\theta},Y^{\theta})\mbox{ and }\rho=L(\eta,\eta).\label{eq:dtinvariance}
\end{align}
for all $X,Y\in TM$. \weg{Hence, for $\eta$-invariant solutions $(L,\sigma,\rho)$, the components $\sigma,\rho$ depend algebraically on $L$.}
\end{lem}
\begin{proof}
Using \eqref{eq:system} and \eqref{eq:connsasaki}, we have that $\mathcal{L}_{\eta}L=0$ if and only if 
$$0=(\mathcal{L}_{\eta}L)(\eta,\eta)=2\sigma(\eta),\,\,\,0=(\mathcal{L}_{\eta}L)(\eta,X^{\theta})=\sigma(X^{\theta})+L(\eta,(JX)^{\theta}),$$
$$0=(\mathcal{L}_{\eta}L)(X^{\theta},Y^{\theta})=L((JX)^{\theta},Y^{\theta})+L(X^{\theta},(JY)^{\theta}).$$
These are the first three equations in \eqref{eq:dtinvariance}. From the invariance of $\sigma$ it follows
$$0=(\mathcal{L}_{\eta}\sigma)(\eta)=\rho-L(\eta,\eta),$$
which is the last equation in \eqref{eq:dtinvariance}. The condition
$$0=(\mathcal{L}_{\eta}\sigma)(X^{\theta})=-L(\eta,X^{\theta})+\sigma((JX)^{\theta})$$
is equivalent to the second equation in \eqref{eq:dtinvariance}. The invariance of $\rho$ is satisfied automatically since $\mathcal{L}_\eta\rho=-2\sigma(\eta)=0$. 
\end{proof}

Next we show
\begin{lem}\label{lem:isom}
There is an isomorphism between the space of solutions $(A,\lambda,\mu)$ of \eqref{eq:hprosystem} on $(M,g,J)$ for $B=-1$ and $\eta$-invariant solutions $(L,\sigma,\rho)$ of \eqref{eq:system} on $(P,h)$.

With respect to the decomposition $TP=\mathbb{R}\,\eta\oplus\mathcal{H}$, the correspondence is given by 
\begin{align}
L=\mu\theta^2+\theta\otimes \lambda(J.)+\lambda(J.)\otimes\theta+A,\,\sigma=\lambda,\,\rho=\mu,\label{eq:A'sasaki}
\end{align}
where we omit the symbol for the pullback of objects from $M$ to $P$.
\end{lem}
\begin{proof}
First let us show, that $(L,\sigma,\rho)$ in \eqref{eq:A'sasaki} defines a solution of \eqref{eq:system}. By direct calculation using the formulas \eqref{eq:connsasaki} for the Levi-Civita connection of $h$ we obtain
$$(\hnabla_{\eta}L)(\eta,\eta)=\eta L(\eta,\eta)-2L(\hnabla_{\eta}\eta,,\eta)=0=h(\eta,\eta)\sigma(\eta)+h(\eta,\eta)\sigma(\eta),$$
$$(\hnabla_{\eta}L)(X^{\theta},\eta)=-L((JX)^{\theta},\eta)=\lambda(X)=h(\eta,X^{\theta})\sigma(\eta)+h(\eta,\eta)\sigma(X^{\theta}),$$
$$(\hnabla_{\eta}L)(X^{\theta},Y^{\theta})=-A(JX,Y)-A(X,JY)=0=h(\eta,X^{\theta})\sigma(Y^{\theta})+h(\eta,Y^{\theta})\sigma(X^{\theta}).$$
Further, using the equations in \eqref{eq:hprosystem} with $B=-1$, we calculate
$$(\hnabla_{Z^{\theta}}L)(\eta,\eta)=Z(\mu)+2\lambda(Z)=0=h(Z^{\theta},\eta)\sigma(\eta)+h(Z^{\theta},\eta)\sigma(\eta),$$
$$(\hnabla_{Z^{\theta}}L)(X^{\theta},\eta)=(\nabla_Z \lambda)(JX)-\mu g(Z,JX)+A(Z,JX)=0=h(Z^{\theta},X^{\theta})\sigma(\eta)+h(Z^{\theta},\eta)\sigma(X^{\theta})$$
and finally
$$(\hnabla_{Z^{\theta}}L)(X^{\theta},Y^{\theta})=g(Z,X)\lambda(Y)+g(Z,Y)\lambda(X)=h(Z^{\theta},X^{\theta})\sigma(Y^{\theta})+h(Z^{\theta},Y^{\theta})\sigma(X^{\theta}).$$
We have shown that $L$ defined in \eqref{eq:A'sasaki} satisfies the first equation in  \eqref{eq:system}. For $\sigma$ we obtain
$$(\hnabla_{\eta}\sigma)(\eta)=0=\rho h(\eta,\eta)-L(\eta,\eta),$$
$$(\hnabla_{\eta}\sigma)(X^{\theta})=-\sigma((JX)^{\theta})=-\lambda(JX)=\rho h(\eta,X^{\theta})-L(\eta,X^{\theta})$$
and
$$(\hnabla_{Z^{\theta}}\sigma)(X^{\theta})=\mu g(Z,X)-A(Z,X)=\rho h(Z^{\theta},X^{\theta})-L(Z^{\theta},X^{\theta}).$$
Thus, $\sigma$ satisfies the second equation in  \eqref{eq:system}. Finally, for $\rho$ we have 
$$\hnabla_{\eta}\rho=0=-2\sigma(\eta)\mbox{ and }\hnabla_{X^{\theta}}\rho=-2\lambda(X)=-2\sigma(X^{\theta}).$$

Now we show that under the correspondence \eqref{eq:A'sasaki}, the $\eta$-invariant solution $(L,\sigma,\rho)$ of \eqref{eq:system} descends to a solution $(A,\lambda,\mu)$ of \eqref{eq:hprosystem} with $B=-1$. Using \eqref{eq:dtinvariance}, we calculate
$$(\nabla_{Z}A)(X,Y)=(\hnabla_{Z^{\theta}}L)(X^{\theta},Y^{\theta})+\omega(Z,X)\sigma((JY)^\theta)+\omega(Z,Y)\sigma((JX)^\theta)$$
$$=h(Z^{\theta},X^{\theta})\sigma(Y^{\theta})+h(Z^{\theta},Y^{\theta})\sigma(X^{\theta})+\omega(Z,X)\sigma((JY)^{\theta})+\omega(Z,Y)\sigma((JX)^{\theta})$$
$$=g(Z,X)\lambda(Y)+g(Z,Y)\lambda(X)+\omega(Z,X)\lambda(JY)+\omega(Z,Y)\lambda(JX),$$
which is the first equation in \eqref{eq:hprosystem}. To verify the second equation, we calculate
$$(\nabla_{Z}\lambda)(X)=(\hnabla_{Z^{\theta}}\sigma)(X^{\theta})=\rho h(Z^{\theta},X^{\theta})-L(Z^{\theta},X^{\theta})=\mu g(Z,X)-A(Z,X).$$
Finally, $\nabla_Z \mu=-2\sigma(Z^{\theta})=-2\lambda(Z)$. This completes the proof of Lemma \ref{lem:isom}.
\end{proof}

Recall that we already have an isomorphism \eqref{eq:isomorphism} between the space of symmetric parallel $(0,2)$-tensors $\hat{A}$ on $(\hat{M},\hat{g},\hat{J})$ and solutions $(L,\sigma,\rho)$ of \eqref{eq:system} on $(P,h)$. To prove Theorem \ref{thm:coneconstruction2}, it remains to show that hermitian symmetric parallel $(0,2)$-tensors $\hat{A}$ correspond to $\eta$-invariant solutions $(L,\sigma,\rho)$. 

Using the definition of $\hat{J}$, it is easy to see that $\hat{A}$ in \eqref{eq:isomorphism} is hermitian, i.e., $\hat{A}(\hat{J}.,\hat{J}.)=\hat{A}$, if and only if $(L,\sigma,\rho)$ satisfies the equations \eqref{eq:dtinvariance}. By Lemma \ref{lem:dtinvariance}, it follows that $\hat{A}$ is hermitian if and only if $(L,\sigma,\rho)$ is $\eta$-invariant. This completes the proof of Theorem \ref{thm:coneconstruction2}.

\section{How  to reduce the investigation of $\mathcal{A}(g,J)$ of dimension $\ge 3$ locally to the investigation of the space of parallel hermitian $(0,2)$-tensors on the conification} 
\label{sec:mainfeatures}

If $B=-1$ in \eqref{eq:hprosystem}, the  reduction was done in the previous section. For our purposes it is sufficiently to assume that the manifold is diffeomorphic to the $2n$-dimensional  ball, the transition ``any ball'' $\longrightarrow$ ``any simply-connected manifold'' will be done in Section \ref{sec:proofthmdegreeglobal}. The goal of the section is to show that on every  neighborhood $U$ that is   diffeomorphic to the ball and  such that the closure of $U$ is compact one can achieve $B=-1$ by replacing the metric by its c-projectively equivalent. If $B\ne  0$, this could be done  by a  scaling of $g$, see the proof of Corollary  \ref{cor:changeofmetric}. If $B=0$, then we do need to change the metric  by a essentially  c-projectively equivalent one:

\begin{lem}\label{lem:changeofmetric}
Let $(M,g,J)$ be a connected Riemannian K\"ahler manifold of real dimension $2n\geq 4$. Assume there exists a solution  $(A, \lambda, \mu)$ 
of \eqref{eq:hprosystem}  with $B=0$ such that $\lambda \ne 0$. 

Then for  every open simply connected subset $U\subseteq M$ with compact closure, there exists  a Riemannian K\"ahler metric $\tilde{g}$ on $U$ 
 that is c-projectively equivalent to $g$, 
 such that  for  any solution  $(\tilde A,\tilde \lambda)$  of the system  \eqref{eq:mainA} for the metric $\tilde g$ 
  there exists  a function $\tilde \mu$ such that  $(\tilde A,\tilde \lambda, \tilde \mu)$
  satisfies \eqref{eq:hprosystem} for the metric   $\tilde{g}$  and such that the corresponding  constant $\tilde{B}$ is different from $0$.
\end{lem}
The proof of Lemma \ref{lem:changeofmetric} will be given in Section \ref{sec:proofoflemma}.

As we already remarked, $D(g,J)$ is the same for all c-projectively equivalent metrics. As an direct application of Lemma \ref{lem:changeofmetric} we obtain 
\begin{cor}\label{cor:changeofmetric}
Let $(M,g,J)$ be a connected Riemannian K\"ahler manifold of real dimension $2n\geq 4$ and of degree of mobility $D(g,J)\geq 3$. Suppose there exists at least one metric c-projectively equivalent to $g$ and not affinely equivalent to it. Then, on each open simply connected subset $U\subseteq M$ with compact closure, the degree of mobility $D(g_{|U},J_{|U})$ is equal to the dimension of the space of solutions of \eqref{eq:hprosystem} with $B=-1$ for a certain positively or negatively definite K\"ahler metric $\tilde{g}$ on $U$ that is c-projectively equivalent to $g$. 
\end{cor}
\begin{proof}
By Lemma \ref{lem:changeofmetric}, on every open simply-connected subset $U$ with the required properties, we can find a Riemannian K\"ahler metric $g'$, c-projectively equivalent to $g$, such that $\mathcal{A}(g,J)$ is isomorphic to the space of solutions of the system \eqref{eq:hprosystem} for $g'$ with a certain constants $B'\neq 0$. For the rescaled metric $\tilde{g}= -B'g'$, the system \eqref{eq:hprosystem} holds now with a constant $\tilde{B}=-1$. Depending on the sign of $B'$, the new metric $\tilde{g}$ is either positively or negatively definite.
\end{proof}

\subsection{Proof of Lemma \ref{lem:changeofmetric}: the constant $B$ in the system \eqref{eq:hprosystem} can be made non-zero by an arbitrary small  change of the metric in the c-projective class}\label{sec:proofoflemma}

The proof of Lemma \ref{lem:changeofmetric} is divided into several steps. First we show how the constant $B$ changes if one chooses another metric in the c-projective class.
\begin{lem}\label{lem:changeofmetric1}
Let $(M,g,J)$ be a Riemannian K\"ahler manifold of real dimension $2n\geq 4$ and let $D(g,J)\geq 3$. Suppose $\tilde{g}$ is c-projectively equivalent to $g$ and let $A=A(g,\tilde{g})\in \mathcal{A}(g,J)$ be given by formula \eqref{eq:defA}. Let $\lambda$ and $\mu$ be the $1$-form and function respectively such that $(A,\lambda,\mu)$ constitutes a solution of the system \eqref{eq:hprosystem} for $g$ with constant $B$ and let $\Lambda=g^{-1}\lambda$. Then the constant $\tilde{B}$ in the system \eqref{eq:hprosystem} for $\tilde{g}$ is given by
\begin{align}
\tilde{B}=(\mathrm{det}A)^{\frac{1}{2}}(g(A^{-1}\Lambda,\Lambda)-\mu).\label{eq:trafoB}
\end{align}
\end{lem}
\begin{proof}
Let us view the tensor $A=A(g,\tilde{g})$ in equation \eqref{eq:defA} equivalently as a $(1,1)$-tensor $g^{-1}A$ by raising the ``left index'' of $A$ by contraction with the inverse metric $g^{-1}$. To simplify notation, we will denote $g^{-1}A$ again by $A$ such that the equation \eqref{eq:mainA} now reads
\begin{align}
\nabla_X A=g(.,X)\Lambda+g(.,\Lambda)X+g(.,JX)J\Lambda+g(.,J\Lambda)JX.\label{eq:mainAalternative}
\end{align}
From the defining equation \eqref{eq:defA} it follows that $A^{-1}=A(\tilde{g},g)$, thus $\tilde{A}=A^{-1}$ is a solution of \eqref{eq:mainAalternative} written down in terms of $\tilde{g}$.
The corresponding vector field $\tilde{\Lambda}$ can be expressed in terms of $A$ and $\Lambda$. Indeed, a straight-forward calculation, using \eqref{eq:lchpro}, \eqref{eq:defA}, \eqref{eq:mainAalternative} and \eqref{eq:relationphilambda} yields
$$\tilde{\nabla}_{X}\tilde{A}=-(\mathrm{det}A)^{\frac{1}{2}}(\tilde{g}(.,X)A^{-1}\Lambda+\tilde{g}(.,A^{-1}\Lambda)X+\tilde{g}(.,JX)JA^{-1}\Lambda+\tilde{g}(.,JA^{-1}\Lambda)JX).$$
Comparing this with the expected form of equation \eqref{eq:mainAalternative} for $\tilde{g}$, we see that
\begin{align}
\tilde{\Lambda}=-(\mathrm{det}A)^{\frac{1}{2}}A^{-1}\Lambda.\label{eq:trafolambda}
\end{align}
We will use this equation to calculate the second equation in the system \eqref{eq:hprosystem} for $\tilde{g}$. First we note that 
$$\nabla_X (\mathrm{det}A)^{\frac{1}{2}}=2(\mathrm{det}A)^{\frac{1}{2}}g(A^{-1}\Lambda,X).$$
Using this together with \eqref{eq:trafolambda}, \eqref{eq:lchpro}, \eqref{eq:relationphilambda} and \eqref{eq:hprosystem}, a straight-forward calculation yields
$$\tilde{\nabla}_X \tilde{\Lambda}=(\mathrm{det}A)^{\frac{1}{2}}(g(A^{-1}\Lambda,A^{-1}\Lambda)-B)X+(\mathrm{det}A)^{\frac{1}{2}}(g(A^{-1}\Lambda,\Lambda)-\mu)\tilde{A}X.$$
Comparing this with the expected form of the second equation in \eqref{eq:hprosystem} for $\tilde{g}$, we see that $\tilde{B}$ is given by \eqref{eq:trafoB} as we claimed.
\end{proof}

We consider the case when for the metric $g$ the constant $B$ in \eqref{eq:hprosystem} is vanishing. The third equation in \eqref{eq:hprosystem} shows that the function $\mu$ is necessarily a constant. Next we show that we can always find a solution $(A,\lambda,\mu)$ of \eqref{eq:hprosystem} such that $\mu\neq 0$.
\begin{lem}\label{lem:changeofmetric2}
Let $(M,g,J)$ be a connected K\"ahler manifold of real dimension $2n\geq 4$ and of degree of mobility $D(g,J)\geq 3$. Suppose the system \eqref{eq:hprosystem} holds for $B=0$ and that at least one metric c-projectively equivalent to $g$ is not affinely equivalent to it. Then on every open simply connected subset $U\subseteq M$, we can find a solution $(A,\lambda,\mu)$ of \eqref{eq:hprosystem} such that $\mu\neq 0$.
\end{lem}
\begin{proof}
Recall from Remark \ref{rem:affine} that if $\tilde{g}$ is c-projectively equivalent to $g$ but not affinely equivalent to it, the $1$-form $\lambda$ corresponding to $A=A(g,\tilde{g})\in \mathcal{A}(g,J)$ is not identically zero. Let us work with this solution $A$. The equations in \eqref{eq:hprosystem} show that for the corresponding $1$-form $\lambda$ we have
$$\nabla_X \lambda =\mu g(X,.)$$
for a certain constant $\mu$. 

Suppose $\mu=0$, i.e. $\nabla_X \lambda=0$. Consider the $1$-form $Ag^{-1}\lambda$ (where as always both $g^{-1}:T^{*}M\rightarrow TM$ and $A:TM\rightarrow T^{*}M$ are viewed as bundle morphisms). Calculating its covariant derivative using \eqref{eq:mainA}, we obtain
\begin{align}
\nabla_X(Ag^{-1}\lambda)=(\nabla_X A)g^{-1}\lambda=\lambda(g^{-1}\lambda)g(X,.)+\lambda(X)\lambda-\lambda(JX)\lambda(J.).\label{eq:oneformderivative}
\end{align}
Recall that $\lambda$ is the differential of a function, i.e. $\lambda=\nabla f$ for a certain function $f$. On the other hand, on the open neighborhood $U$ also the $1$-form $\lambda(J.)$ is the differential of a function $f':U\rightarrow \mathbb{R}$. This follows from the fact that $\lambda(J.)$ is parallel (and hence closed) and $U$ is simply connected. Let us set $c=\lambda(g^{-1}\lambda)$ (which is a non-zero constant) and define the $1$-form
$\sigma=Ag^{-1}\lambda-f\lambda+f'\lambda(J.).$
It follows from \eqref{eq:oneformderivative}, that
$$\nabla_X\sigma=c g(X,.).$$
On the other hand, it is straight-forward to check that the symmetric hermitian $(0,2)$-tensor
$$\tilde{A}=\sigma\otimes \sigma+ \sigma(J.)\otimes \sigma(J.)$$
satisfies \eqref{eq:mainA}. The corresponding $1$-form $\tilde{\lambda}$ is given by $c\sigma$ and satisfies
$\nabla_X \tilde{\lambda}=c^2 g(X,.).$
Thus $(\tilde{A},\tilde{\lambda},\tilde{\mu}=c^2)$ is the desired solution of \eqref{eq:hprosystem} with $B=0$ but $\tilde{\mu}\neq 0$.
\end{proof}

Now we are able to give the proof of Lemma \ref{lem:changeofmetric}. Let us suppose that $B=0$ and let $U$ be an open simply connected subset with compact closure. By Lemma \ref{lem:changeofmetric2}, we can find a solution $(A,\lambda,\mu)$ of \eqref{eq:hprosystem} on $U$ such that $\mu\neq 0$ and after rescaling we can suppose that $\mu=1$. 

For arbitrary real numbers $t$, we define the triple 
$$A(t)=t(\lambda\otimes\lambda+\lambda(J.)\otimes \lambda(J.))+g,\,\,\,\lambda(t)=t\lambda,\,\,\,\mu(t)=t\mu=t.$$
Obviously the triple $(A(t),\lambda(t),\mu(t))$ is a solution of \eqref{eq:hprosystem} for $g$ with $B=0$. Moreover, since $U$ has compact closure, we find $\epsilon>0$ such that for all $t\in (-\epsilon,\epsilon)$ the solution $A(t)$ is non-degenerate on $U$ and the metric $\tilde{g}_t$ in the c-projective class of $g$ which corresponds to $A(t)$ (that is, $\tilde{g}_t$ is defined by $A(g,\tilde{g}_t)=A(t)$ in equation \eqref{eq:defA}) is positively definite. 

Using Lemma \ref{lem:changeofmetric1}, we see that the constant $\tilde{B}_t$ in the system \eqref{eq:hprosystem} for $\tilde{g}_t$ is given by
$$\tilde{B}(t)=(\mathrm{det}A(t))^{\frac{1}{2}}(g(A(t)^{-1}\Lambda(t),\Lambda(t))-\mu(t))=(\mathrm{det}A(t))^{\frac{1}{2}}(t^2g(A(t)^{-1}\Lambda,\Lambda)-t),$$
where as usual $\Lambda=g^{-1}\lambda$. We want to show that $\tilde{B}(t)$ is non-zero for some $t_0\in (-\epsilon,\epsilon)$. Since $(\mathrm{det}A(t))^{\frac{1}{2}}$ is positive anyway it suffices to show that 
$$f(t)=t^2g(A(t)^{-1}\Lambda,\Lambda)-t$$
is non-zero for some $t_0\in (-\epsilon,\epsilon)$. Taking the derivative of $f$ when $t=0$, we obtain $\tfrac{d f}{dt}(0)=-1,$ thus, there is some $t_0\in (-\epsilon,\epsilon)$ such that $\tilde{B}(t_0)$ is non-zero. The metric $\tilde{g}_{t_0}$ has the properties as required in Lemma \ref{lem:changeofmetric}. This completes the proof of Lemma \ref{lem:changeofmetric}.

\begin{rem} Though we explicitly used in the proof of Lemma \ref{lem:changeofmetric2} 
that the metric $g$ is Riemannian, Lemma \ref{lem:changeofmetric2}  remains true for metrics of arbitrary signature (but the proof in  arbitrary signature is  longer and uses nontrivial results of \cite{EMN}). 
\end{rem}

\subsection{Conification of Einstein manifolds} \label{-4} 
Suppose $(M^{2n\ge 4}, g, J)$ is a Kähler-Einstein manifold of arbitrary signature. We assume that the symplectic form $\omega= g(.,J.)$ is exact so that we can consider the conification $(\hat M^{2n+2}, \hat g, \hat J)$ introduced in Section \ref{sec:construction}. Our goal is to show that the investigation of solutions of \eqref{eq:mainA} on $(M,g,J)$ reduces to the investigation of parallel tensors on the conification $(\hat{M},\hat{g},\hat{J})$ for a Ricci flat metric $\hat{g}$. We start with the following technical statement.  

\begin{lem}\label{lem:scalarB}
Let $(M,g,J)$ be a K\"ahler-Einstein manifold of real dimension $2n\geq 4$. Suppose there exists a solution $(A, \lambda, \mu) $ of
\eqref{eq:hprosystem} such that $\lambda\ne 0$. 
 Then, 
\begin{align}
B=-\frac{\mathrm{Scal}(g)}{4n(n+1)},\label{eq:scalarB}
\end{align}
where $\mathrm{Scal}(g)$ is the scalar curvature of $g$. 
\end{lem}
\begin{proof}
Take  a solution $A$ of \eqref{eq:hprosystem} such that  $\lambda\ne 0$. As usual we denote  $\Lambda=g^{-1}\lambda$. It is known (see \cite[Proposition 3]{ApostolovI}, \cite[Corollary 3]{FKMR} or equation (13) and the sentence below in \cite{DomMik1978}) that  $J\Lambda$ is a Killing vector field which in particular implies that  $\lambda$ is non-zero in every point of an open and everywhere dense subset. We define the function $\sigma=\mathrm{trace}\nabla\Lambda$. From the second equation in \eqref{eq:hprosystem}, we see that $\sigma$ and $\mu$ are related by $\sigma=2n\mu+B\mathrm{trace}(A).$ Thus, taking the covariant derivative of this equation and inserting the third equation of \eqref{eq:hprosystem}, we obtain 
\begin{align}
\nabla\sigma=2n\nabla\mu+4B\lambda=4B(n+1)\lambda.\label{eq:musigma}
\end{align}
 
On the other hand, since $J\Lambda$ is Killing, we have the identity $\nabla_X\nabla\Lambda=-JR(X,J\Lambda)$. Using the symmetries of the curvature tensor $R$ of the K\"ahler metric $g$, we have $\mathrm{trace}(JR(X,JY))=2\mathrm{Ric}(g)(X,Y)$. Combining the last two equations yields $\nabla\sigma=-2\mathrm{Ric}(g)(.,\Lambda)$ and inserting this into \eqref{eq:musigma}, we have
$$-\mathrm{Ric}(g)(.,\Lambda)=2B(n+1)\lambda.$$
Since  $g$ is K\"ahler-Einstein, that is $\mathrm{Ric}(g)=\frac{\mathrm{Scal}(g)}{2n}g$, we evidently have \eqref{eq:scalarB}.
\end{proof}

\begin{lem} \label{lem:eins}  Let $(M,g,J)$ be a K\"ahler-Einstein manifold of real dimension $2n\geq 4$ and let  the symplectic form $\omega= g(., J.)$ be exact. Assume
 $ {\mathrm{Scal}(g)}= 4n(n+1)$. Then, the conification of 
 $(M,g,J)$  is Ricci flat. Moreover, if a conification of a certain K\"ahler-Einstein manifold $(M^{\ge 4},g,J)$  is Ricci-flat, then $g$ is Einstein with scalar  curvature  $ {\mathrm{Scal}(g)}= 4n(n+1)$. 
\end{lem} 
\begin{proof} 
By direct calculation, using the formulas \eqref{eq:LCconncone}, we obtain that the curvature tensor $\hat{R}$ of $(\hat{M},\hat{g},\hat{J})$ is given by the formulas
\begin{align}
\hat{R}(.,.)\xi=\hat{R}(.,.)J\xi=0,\label{eq:curvcon0}
\end{align}
where $\xi=r\partial_r$ is the cone vector field on $(\hat{M},\hat{g},\hat{J})$, and 
\begin{align}
\hat{R}(X^\theta,Y^\theta)Z^\theta=(R(X,Y)Z-4H(X,Y)Z)^\theta,\label{eq:curvcon}
\end{align}
where $R$ denotes the curvature tensor of $(M,g,J)$,
$$H(X,Y)Z=\frac{1}{4}(g(Z,Y)X-g(Z,X)Y+\omega(Z,Y)JX-\omega(Z,X)JY+2\omega(X,Y)JZ)$$
is the algebraic curvature tensor of constant holomorphic  curvature equal to one and $X^\theta$ denotes the horizontal lift of tangent vectors $X\in TM$ to the distribution $\mathcal{H}=\mathrm{span}\{\xi,J\xi\}^\perp\subseteq T\hat{M}$ (see also Section \ref{sec:construction} for the notation). Having these formulas, Lemma \ref{lem:eins} follows from simple linear algebra:

From \eqref{eq:curvcon0} it is clear that $\mathrm{Ric}(\hat{g})(\xi,.)=\mathrm{Ric}(\hat{g})(J\xi,.)=0$, where $\xi$ denotes the cone vector field on $\hat{M}$. A straight-forward calculation using \eqref{eq:curvcon} yields
$$\mathrm{Ric}(\hat{g})(X^\theta,Y^\theta)=r^2\left(\mathrm{Ric}(g)(X,Y)-2(n+1)g(X,Y)\right).$$
implying  that if  $ {\mathrm{Scal}(g)}= 4n(n+1)$ then 
$$\mathrm{Ric}(\hat{g})(X^\theta,Y^\theta)=r^2\left(\mathrm{Ric}(g)(X,Y)-\frac{\mathrm{Scal}(g)}{2n}g(X,Y)\right)=0 $$
so $\hat{g}$ is Ricci flat,  and if $\hat{g}$ is Ricci flat then $g$ is Einstein with $ {\mathrm{Scal}(g)}= 4n(n+1)$.   
\end{proof}

\begin{lem}\label{lem:einsteincproeinstein}
Let $(M,g,J)$ be a K\"ahler-Einstein manifold of real dimension $2n\geq 4$.
Suppose there exists a solution $(A, \lambda, \mu) $ of
\eqref{eq:hprosystem} such that $\lambda\ne 0$. 
 Then, every metric $\tilde{g}$, c-projectively equivalent to $g$, is also K\"ahler-Einstein.
\end{lem}
\begin{proof}  
Let $A=A(g,\tilde{g})$ be the solution of \eqref{eq:mainA} given by \eqref{eq:defA}. In the second equation $\nabla\lambda=\mu g+B A$ of \eqref{eq:hprosystem}, we express $\lambda$ in terms of the $1$-form $\Phi$ by using the relation \eqref{eq:relationphilambda}. A straight-forward calculation yields
$$(\mu-g(A^{-1}\Lambda,\Lambda))gA^{-1}+Bg=-\nabla \Phi+\Phi\otimes\Phi-\Phi(J.)\otimes\Phi(J.),$$
where $A$ is equivalently viewed as a $(1,1)$-tensor and $\Lambda=g^{-1}\lambda$ as usual. Using $\tilde{g}=(\mathrm{det}\,A)^{-\frac{1}{2}}gA^{-1}$ and the transformation rule \eqref{eq:trafoB} for $B$, we can rewrite this into the form 
\begin{align}
Bg-\tilde{B}\tilde{g}=-\nabla \Phi+\Phi\otimes\Phi-\Phi(J.)\otimes\Phi(J.),\label{eq:kaehlereinstein}
\end{align}

On the other hand, using the transformation rule \eqref{eq:lchpro} for the Levi-Civita connections of the two metrics, it is straight-forward to show and well-known (see \cite[equation (1.11)]{Tashiro1956}) that the Ricci-tensors corresponding to $g$ and $\tilde{g}$ are related by
\begin{align}
\mathrm{Ric}(\tilde{g})=\mathrm{Ric}(g)-2(n+1)(\nabla\Phi-\Phi\otimes\Phi+\Phi(J.)\otimes\Phi(J.)).\label{eq:riccitafo}
\end{align}

Combining this equation with \eqref{eq:kaehlereinstein}, we obtain
$$\mathrm{Ric}(\tilde{g})+2(n+1)\tilde{B}\tilde{g}=\mathrm{Ric}(g)+2(n+1)Bg.$$
Since the assumptions of Lemma \ref{lem:scalarB} are satisfied and $g$ is Einstein, we see from formula \eqref{eq:scalarB} that the right-hand side of the last equation is vanishing identically. Then, $\tilde{g}$ is Einstein.
\end{proof}

Combining Lemma \ref{lem:changeofmetric}  with Lemmas \ref{lem:scalarB} and \ref{lem:einsteincproeinstein}, we obtain 

\begin{cor}\label{cor:changeofmetriceinstein}
Let $(M,g,J)$ be a  Riemannian K\"ahler-Einstein  manifold of real dimension $2n\geq 4$ and of degree of mobility $D(g,J)\geq 3$.
 Suppose  there exists a solution $(A, \lambda, \mu) $ of \eqref{eq:hprosystem} such that $\lambda\ne 0$.

Then, on each open simply connected subset $U\subseteq M$ with compact closure, the degree of mobility $D(g_{|U},J_{|U})$ is equal to the dimension of the space of solutions of \eqref{eq:hprosystem} with $B= -\frac{\mathrm{Scal}(\tilde{g})}{4n(n+1)}= -1$ for a certain positively or negatively definite K\"ahler-Einstein metric $\tilde{g}$ on $U$ that is c-projectively equivalent to $g$. 
\end{cor}

It follows from Corollary \ref{cor:changeofmetriceinstein}, Theorem \ref{thm:coneconstruction2} and Lemma \ref{lem:eins} that (at least in the local setting) we reduced the study of the degrees of mobility of $2n$-dimensional Kähler-Einstein Riemannian metrics to the study of the possible dimensions of the space of parallel hermitian symmetric $(0,2)$-tensors of Ricci-flat cone K\"ahler manifolds $(\hat{M},\hat{g},\hat{J})$ of dimension $2(n+1)$ where $\hat{g}$ is positively definite or has signature $(2,2n)$.

In the proof of  Theorem \ref{thm:degreeeinstein}  we will  need one more observation.
\begin{lem}\label{lem:ricciflatcone}
Let $(\hat{M}=\mathbb{R}_{>0}\times P,\hat{g}=dr^2+r^2 h,\hat{J})$ be a K\"ahler manifold which is the cone over an $(2n+1)$-dimensional pseudo-Riemannian manifold $(P,h)$. 

\begin{enumerate}
\item If $\mathrm{dim}\,\hat{M}<6$ and $\hat{g}$ is Ricci flat, then $\hat{g}$ is flat.  
\item Let $\hat{g}$ have signature $(2,2n)$. If $\mathrm{dim}\,\hat{M}<8$, $\hat{g}$ is Ricci flat and $X$ is a non-zero parallel null-vector field on $\hat{M}$, then $\hat{g}$ is flat.  
\end{enumerate}
\end{lem}
\begin{proof}
(1) Using \eqref{eq:LCcone} it is straight-forward to calculate that the curvature tensor $\hat{R}$ of $\hat{g}$ is given by the formulas
\begin{align}
\hat{R}(.,.)\partial_r=0\mbox{ and }\hat{R}(X,Y)Z=R(h)(X,Y)Z-(h(Z,Y)X-h(Z,X)Y),\label{eq:curvaturecone}
\end{align}
where $X,Y,Z\in TP$ and $R(h)$ is the curvature tensor of $h$. Then, calculating the Ricci tensor of $\hat{g}$, it is straight-forward to see that if $\hat{g}$ is Ricci flat, $h$ is Einstein with scalar curvature $\mathrm{Scal}(h)=(\mathrm{dim}\,P)(\mathrm{dim}\,P-1)$. If in addition $\mathrm{dim}\,\hat{M}<6$, we have that $\mathrm{dim}\,P<4$ and therefore $h$ has constant  curvature equal to $1$. Inserting this back into \eqref{eq:curvaturecone}, we obtain that $\hat{R}=0$ as we claimed.

(2) Let $\xi=r\partial_r$ denote the cone vector field on $\hat{M}$ and let $X$ be the parallel non-zero null vector field. Since $\hat{\nabla}\xi=\Id$, we obtain that $\hat{g}(X,\xi)\neq 0$ on a dense and open subset of $\hat{M}$. Indeed, if $\hat{g}(X,\xi)=0$ on an open subset $U$ we can take the covariant derivative of this equation in the direction of $Y\in T\hat{M}$ to obtain that $\hat{g}(X,Y)=0$ in every point $p\in U$ for all $Y\in T_p \hat{M}$. This implies $X=0$ on $\hat{M}$ which contradicts our assumption. By similar arguments one also obtains that $\hat{g}(JX,\xi)\neq 0$ on a dense and open subset of $\hat{M}$. Let us work in a point $p\in \hat{M}$ such that $\hat{g}(X,\xi)\neq 0$ and $\hat{g}(JX,\xi)\neq 0$ at $p$. We suppose that $\mathrm{dim}\,\hat{M}=6$ and show that $\hat{g}$ is flat. 

It is an easy exercise to show that there exist a basis $X,Y,Z,JX,JY,JZ$ of $T_p \hat{M}$ in which $\hat{g}$ takes the form 
\begin{align}
\hat{g}=\left(\begin{array}{cccccc}0&1&0&0&0&0\\1&0&0&0&0&0\\0&0&1&0&0&0\\0&0&0&0&1&0\\0&0&0&1&0&0\\0&0&0&0&0&1\end{array}\right)\nonumber
\end{align}
and such that $\mathrm{span}\{X,Y,\hat{J}X,\hat{J}Y\}=\mathrm{span}\{X,\xi,\hat{J}X,\hat{J}\xi\}$. Hence any endomorphism of $T_p \hat{M}$ that commutes with $\hat{J}$ and vanishes on $\xi$ and $X$ has to vanish on $Y$ as well. This holds in particular for the curvature endomorphisms $\hat{R}=\hat{R}(u,v):T_p\hat{M}\rightarrow T_p\hat{M}$ for every pair of vectors $u,v\in T_p \hat{M}$. Since $\hat{R}$ commutes with $\hat{J}$ and is skew-symmetric with respect to $\hat{g}$ it takes the form
\begin{align}
\hat{R}=\left(\begin{array}{cccccc}a&0&b&-A&-B&-C\\0&-a&c&-D&-A&-E\\-c&-b&0&-E&-C&-F\\A&B&C&a&0&b\\D&A&E&0&-a&c\\E&C&F&-c&-b&0\end{array}\right).\nonumber
\end{align}
in the basis $X,Y,Z,JX,JY,JZ$ from above. Since $\hat{R}$ vanishes on $X,Y,JX,JY$ it implies that $a=b=c=A=B=C=D=E=0$. Furthermore, using the condition that $\hat{g}$ is Ricci flat, i.e., $\mathrm{trace}(\hat{J}\hat{R})=0$ yields $F=0$. Thus the curvature tensor $\hat{R}$ is vanishing in every point of a dense and open subset, implying that $\hat{g}$ is flat as we claimed.

\weg{Let $V=T_p \hat{M}$ and denote by $V^\mathbb{C}=V\otimes_{\mathbb{R}}\mathbb{C}$ its complexification which decomposes as 
$$V^{\mathbb{C}}=V^{1,0}\oplus V^{0,1},$$
where $V^{1,0},V^{0,1}$ denote the eigenspaces of $J$ corresponding to the eigenvalues $i$ and $-i$ respectively. With respect to this decomposition of $V^\mathbb{C}$ a vector $u\in V^\mathbb{C}$ decomposes as $u=u^{1,0}+u^{0,1}$, where $u^{1,0}=\tfrac{1}{2}(u-iJu)$ and $u^{0,1}=\tfrac{1}{2}(u+iJu)$. Consider the hermitian inner product
$$\langle u+iv,w+iz\rangle=g(u+iv,w-iz)=g(u,w)+g(v,z)+i(g(v,w)-g(u,z))$$
on $V^\mathbb{C}$ with respect to which $V^{1,0}$ and $V^{0,1}$ are orthogonal to each other and which has signature $(1,2)$ when restricted to $V^{1,0}$. Note that we have $\langle X^{1,0},X^{1,0}\rangle =0$ and $\langle X^{1,0},\xi^{1,0}\rangle \neq 0$. By rescaling $X$, we can achieve that $\langle X^{1,0},\xi^{1,0}\rangle = 1$. Introducing a new vector $Y^{1,0}=\xi^{1,0}+\alpha X^{1,0}$ and choosing $\alpha$ appropriately (in particular, we can choose $\alpha$ to be real), we obtain $\langle X^{1,0},Y^{1,0}\rangle = 1$ and $\langle Y^{1,0},Y^{1,0}\rangle = 0$. Let $Z^{1,0}$ be a unit vector orthogonal to $X^{1,0}$ and $Y^{1,0}$. Since without loss of generality, we can assume that $Y,Z$ are real vectors, we have found a basis $X,Y,Z,JX,JY,JZ$ of $T_p M$}

\end{proof}

\section{Proof of Theorems \ref{thm:degree} and \ref{thm:degreeeinstein}}
\label{sec:proofs}

\subsection{Proof of the first statement  of the Theorems \ref{thm:degree} and \ref{thm:degreeeinstein} in the local situation} 
\label{sec:proofthmdegree}
Let $(M,g,J)$ be a K\"ahler manifold of real dimension $2n\geq 4$. Our goal is to show that for every open simply connected subset $U\subseteq M$ with compact closure and the property that the K\"ahler form $\omega$ is exact on $U$, the degree of mobility $D(g_{|U},J_{|U})$ is given by one of the values either in the list of Theorem \ref{thm:degree} for a general metric or in the list of Theorem \ref{thm:degreeeinstein} under the additional assumption that the metric is Einstein. We will prove this simultaneously. 

By Corollary \ref{cor:changeofmetric} and Theorem \ref{thm:coneconstruction2}, the number $D(g_{|U},J_{|U})$ is precisely the dimension of the space of parallel hermitian symmetric $(0,2)$-tensors on the conification $(\hat{U},\hat{\tilde{g}},\hat{J})$ of $(U,\tilde{g}_{|U},J_{|U})$, where $\tilde{g}$ is a certain metric on $U$ that is c-projectively equivalent to $g$. Moreover, since $\tilde{g}$ is either positively or negatively definite, the metric $\hat{\tilde{g}}$ will be either positively definite or has signature $(2,2n)$. We also know  in view of Lemma \ref{lem:einsteincproeinstein}
that if the metric $g$ is Einstein,  then the  metric $ {\tilde{g}}$ is also Einstein so  the metric $\hat{\tilde{g}}$ is Ricci-flat. 

To avoid cumbersome notations, we will drop the ``hat'' and the ``tilde'' in the notation for the conification. The local version of the Theorems \ref{thm:degree} and \ref{thm:degreeeinstein} that we are going to prove in this  section reads

\begin{thm}\label{thm:degreelocal}
Let $(M,g,J)$ be a simply connected K\"ahler manifold of real dimension $2n+2\geq 6$ which is a cone over a $(2n+1)$ dimensional manifold. Further, let $g$ be either positively definite or have signature $(2,2n)$. Then, the dimension $D$ of the space of parallel hermitian symmetric $(0,2)$-tensors is given by one of the values in the list of Theorem \ref{thm:degree}.  Moreover, if the metric $g$ is Ricci-flat, then the dimension $D$ of the space of parallel hermitian symmetric $(0,2)$-tensors is given by one of the values in the list of Theorem \ref{thm:degreeeinstein}.
\end{thm}

The proof of Theorem \ref{thm:degreelocal} in the case when the metric $g$ is positively definite is more simple than in the case when the signature is $(2,2n)$. Moreover, the arguments for the proof when $g$ is positively definite are implicitly contained in the proof when the signature is $(2,2n)$. We therefore restrict to the latter case and assume that $g$ has signature $(2,2n)$ in what follows; the Riemannian signature is explained in Section \ref{sec:ideas} and 
we leave it as an easy exercise. 

\begin{proof} 
Let $p$ be an arbitrary point in $M$. Consider a maximal orthogonal holonomy decomposition of $T_{p}M$. 
\begin{align}
T_pM=T_0\otimes T_1\otimes...\otimes  T_\ell\label{eq:decomptangenttrivial}. 
\end{align}
Here $T_0$ is a nondegenerate subspace of $T_pM$ such that the holonomy group acts trivially and such that it is $J$-invariant,
 and $T_i$ for $i\ge 1$ are $J$-invariant nondegenerate subspace invariant w.r.t. the action of the holonomy group. We assume that the decomposition is maximal in the sense that no $T_i$, $i\ge 1$ has a holonomy-invariant nontrivial  nondegenerate subspace and  therefore can not be decomposed further.     The existence of such a decomposition is standard and follows for example from \cite{Wu}. 

If in addition the initial manifold is Ricci-flat, then the restriction of the curvature tensor to each $T_i$ is also Ricci-flat.

It is well known that symmetric hermitian {\it parallel} $(0,2)$-tensor fields  on $M$ are in one-one correspondence with symmetric hermitian $(0,2)$-tensors on $T_pM$ that are invariant w.r.t. the action of the holonomy group. 
As it was shown in \cite[Theorem 5]{FedMat}, every  symmetric  {\it holonomy-invariant} $(0,2)$-tensor  on $T_pM$ has the form
 \begin{equation}\label{presentation1}
\sum_{\alpha, \beta =1}^{2k}c_{\alpha \beta } \tau_\alpha \otimes  \tau_\beta + C_1 g_1 + ...+ C_\ell  g_\ell.  
\end{equation}
Here $\{\tau_i\}_{i=1,...,2k}$ is a basis in the subspace of $T^*M$ consisting of those elements that are invariant w.r.t. the holonomy group, and $g_i$, $i=1,..., \ell $  is the restriction of $g$ to $T_i$ viewed as $(0,2)$-tensors on $T_pM$.  Note that in the case of indefinite  signature  the number  $k$   must not coincide with $k_0:= \tfrac{1}{2}\dim T_0$, since it might exist a  light-like holonomy-invariant vector such that it is orthogonal to all vectors from  $T_0$. In the signature $(2,2n)$ we have $k=k_0$ or $k=k_0+1$. The coefficients  $c_{\alpha\beta}$ satisfy $c_{\alpha\beta}=c_{\beta\alpha}$ so $(c_{\alpha\beta})$ is a symmetric matrix. Our  assumption that the parallel tensor is  hermitian implies that the matrix $c_{\alpha\beta}$ is hermitian.

Clearly, the dimension of the space of the tensors of the form \eqref{presentation1} is the number of free parameters $c_{\alpha\beta}$, $C_i$.  
It is well known that the space of symmetric hermitian $2k\times 2k$ matrices has dimension $k^2$  so the first term of \eqref{presentation1} gives us $k^2$ dimensions and  we obtain $ k^2 + \ell$  in total which is as we claimed.
Our goal  is to show that $k$ and $\ell$ satisfy the restrictions in the Theorems \ref{thm:degree} and \ref{thm:degreeeinstein}.

Suppose $\ell =0$  (that is $g$ is flat and hence, the initial $2n$-dimensional metric has constant holomorphic curvature). Then, $k_0=n+1$ and we obtain that the dimension of the space of the parallel hermitian tensors is $(n+1)^2$ as we want. 
 
Suppose $\ell\ge 1 $ and take $i\ge 1$. By \cite[Lemma 2]{FedMat}, the dimension of $T_i$ is $\ge 3$ and since it is even, we have $\dim(T_i)\ge 4$.  Moreover, under the additional assumption that $T_i$ is Ricci-flat, we have $\dim(T_i)\ge 6$ by Lemma \ref{lem:ricciflatcone}.

Suppose now $T_i$ with $i\ge 1$ contains a nonzero holonomy-invariant vector. Let us show that then the dimension of this  $T_i$ is $\ge 6$.  
Let us denote this vector by $v$. Note that the vector $Jv$ is also holonomy-invariant and any linear combination of $v$ and $Jv$ is light-like since otherwise there would exist a nontrivial holonomy-invariant nondegenerate (two-dimensional) subspace.

We extend $T_i$ and also $v,Jv\in  T_i$ to the whole manifold by parallel translating these objects along all possible ways starting at $p$. The  extension  of $T_i$ is well-defined and gives us an integrable distribution on $M$.  The extensions  of $v$ and $Jv$ are  also well-defined and are parallel vector fields.

It is sufficient to show that under the assumption $\dim(T_i)=4$ the restriction of the curvature to this distribution vanishes, since this will imply in view of the theorem of Ambrose-Singer that the holonomy group acts trivially on $T_i$ which contradicts the assumption that $T_i$  has no nontrivial   holonomy-invariant  nondegenerate subspaces. 
We choose a generic point $q$.  Since the point is generic, then by \cite[Lemma 5 and Lemma 2]{FedMat} there exists a vector $u\in T_i(q)$ such that $g(u,u)\ne 0$ and such that $R(u,.).= 0$.  We consider the basis $\{u, Ju, v, Jv\}$ of $T_i(q)$. This is indeed a basis since the vectors $u$ and $Ju$ (resp. $v$ and $Jv$) are nonproportional and 
therefore  linearly independent and no nontrivial linear  combination of $u$ and $Ju$ can be equal to a nontrivial linear combination of $v$ abd $Jv$ since any linear combination of $v$ and $Jv$ is light-like and any linear combination of $u$ and $Ju$ is not light-like.  
Now, the  vectors $u,v, Ju, Jv$ satisfy the condition $R(u,.).= R(Ju,.).= R(v,.).= R(Jv,.)= 0$. Indeed, $R(u,.).=0$ is essentially the choice of our vector,  $R(Ju,.).=0$ is because the Riemanian curvature of a Kähler metric is $J$-invariant, $ R(v,.).= R(Jv,.)= 0$ is fulfilled because $v$ and $Jv$ are parallel. Thus, $\dim(T_i)\ge 6$.

Let us now suppose that $T_i$ is Ricci flat and contains a nonzero (and therefore light-like) holonomy-invariant vector. Then, it follows from Lemma \ref{lem:ricciflatcone} that $\dim(T_i)\ge 8$.

We obtain that the number $k$ is at most $n-1$ and the number  $\ell$ is at most $\left[\frac{n-k-1}{2}\right]$. Indeed, suppose there exists a nonzero holonomy-invariant  vector contained in one $T_i$ with $i\ge 1$. Without loss of generality we may assume $i=1$. 
 As we explained above,  the dimension of  $T_1$ is $\ge 6$ and the dimension of all other $T_j$ for $j\ge 2$ is at least $4$. The dimension of $T_0$ is $2k- 2$. Then, 
 \begin{equation} 
 \underbrace{\dim(T_0)}_{2k-2} + \underbrace{\dim(T_1)}_{\ge 6}+ \underbrace{\dim(T_2)}_{\ge 4}+ ... +\underbrace{\dim(T_\ell)}_{\ge 4}=2(n+1)\label{-1}\end{equation}
 implying $k\le n-1$ and   $\ell\le \left[\frac{n-k-1}{2}\right]$ as we want. 
 
 Suppose now there exists no nonzero  holonomy-invariant  vector contained in one $T_i$ with $i\ge 1$. Then, $\dim(T_0)=2k$ and the dimension of all   $T_j$ for $j\ge 1$ is at  least  $4$. Here we obtain $\ell\le \left[\frac{n-k-1}{2}\right]$ by the same argument. Indeed, in this case 
  \begin{equation} \label{+1} 
 \underbrace{\dim(T_0)}_{2k } + \underbrace{\dim(T_1)}_{\ge 4}+   ... +\underbrace{\dim(T_\ell)}_{\ge 4}=2(n+1)\end{equation}
 implying  $\ell\le \left[\frac{n-k-1}{2}\right]$ as we want.

  Assume now the metric $\hat g$ is Ricci-flat. Then, each $T_i$ is Ricci flat, so its dimension is $\ge 6$. As we have shown above, if $T_i$ (with $i\ge 1$) 
  contains a  nonzero  holonomy-invariant  vector, then $\dim(T_i)\ge 8$, and the analog of \eqref{-1} looks 
  $$\underbrace{\dim(T_0)}_{2k-2} + \underbrace{\dim(T_1)}_{\ge 8}+ \underbrace{\dim(T_2)}_{\ge 6}+ ... +\underbrace{\dim(T_\ell)}_{\ge 6}=2(n+1).$$
 implying $k\le n-2$  and $\ell\le \left[\frac{n-k-1}{3}\right]$ as we want. 
 If there exists no nonzero  holonomy-invariant  vector contained in one $T_i$ with $i\ge 1$. Then, $\dim(T_0)=2k$, the dimension of all   $T_j$ for $j\ge 1$ is at least $6$ and we obtain  $k\le n-2$ and 
 $\ell\le \left[\frac{n-k-1}{3}\right]$ by the same argument. Indeed, in this case 
  $$
 \underbrace{\dim(T_0)}_{2k } + \underbrace{\dim(T_1)}_{\ge 6}+   ... +\underbrace{\dim(T_\ell)}_{\ge 6}=2(n+1)$$
 implying $k\le n-2$ and  $\ell\le \left[\frac{n-k-1}{3}\right]$ as we want.
 Theorem \ref{thm:degreelocal}  is proved. 
 \end{proof}

\subsection{Proof of the first parts of Theorems \ref{thm:degree} and \ref{thm:degreeeinstein} in the global situation}
\label{sec:proofthmdegreeglobal}
In this section, we complete the proof of the first parts of the Theorems \ref{thm:degree} and \ref{thm:degreeeinstein}. Let $(M,g,J)$ be a simply connected K\"ahler manifold of real dimension $2n\geq 4$. 

We call a subset $U\subseteq M$ a \emph{ball} if $U$ is open, homeomorphic to an open $2n$-ball in $\mathbb{R}^{2n}$ and has compact closure $\bar{U}$. 

Since a ball $U$ satisfies all the assumptions in Corollary \ref{cor:changeofmetric}, Theorem \ref{thm:coneconstruction2} and Theorem \ref{thm:degreelocal}, we obtain that the degree of mobility $D(g_{|U},J_{|U})$ of the restriction of the K\"ahler structure is given by one of the values in the list of Theorem \ref{thm:degree}.

If in addition $g$ is Einstein, it follows from Corollary \ref{cor:changeofmetriceinstein}, Theorem \ref{thm:coneconstruction2} and Theorem \ref{thm:degreelocal} that $D(g_{|U},J_{|U})$ is given by one of the values in the list of Theorem \ref{thm:degreeeinstein}.

Recall from \cite[Proposition 4]{ApostolovI} or \cite[equation (1.3)]{Mikes} that the space $\mathcal{A}(g,J)$ of hermitian symmetric solutions of equation \eqref{eq:mainA} is isomorphic to the subspace $\mathrm{Par}(E,\nabla^E)$ of the space of sections of a certain vector bundle $\pi:E\rightarrow M$ whose elements are parallel with respect to a certain connection $\nabla^E$ on $E$. In particular, we have $D(g,J)=\mathrm{dim}\,\mathrm{Par}(E,\nabla^E)$. 

The next statement will complete the proof of the first parts of the Theorems \ref{thm:degree} and \ref{thm:degreeeinstein}.
\begin{lem}\label{lem:extensiontoglobal}
Let $M$ be a simply connected manifold and let $\pi:E\rightarrow M$ be a vector bundle over $M$ with a connection $\nabla^E$. Let $I\subseteq \mathbb{N}$ be a set of nonnegative 
integers and suppose that for every ball $U\subseteq M$, we have $\mathrm{dim}\,\mathrm{Par}(E_{|U},\nabla^E)\in I$. Then, $\mathrm{dim}\,\mathrm{Par}(E,\nabla^E)\in I$.
\end{lem}
\begin{proof}
\begin{figure}
  \includegraphics[width=.6\textwidth]{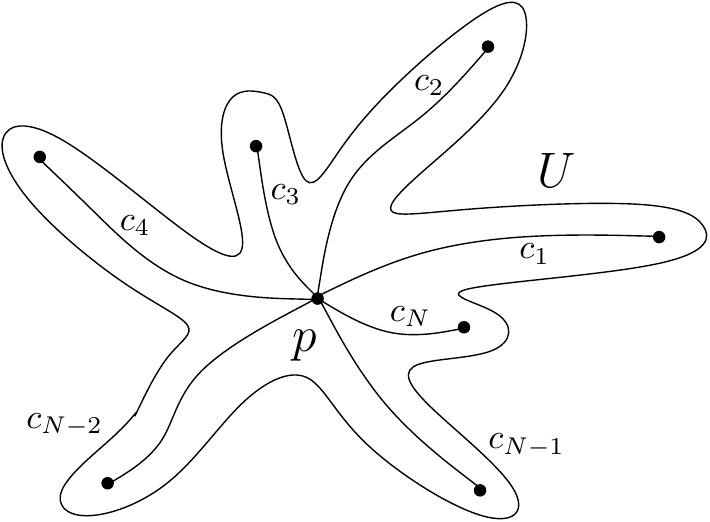}
  \caption{We can choose a tubular neighborhood $U$ of the union $\bigcup_{i=1}^N c_i([0,1])$ of the curves $c_1,...,c_N:[0,1]\rightarrow M$ such that $U$ is a ball.}\label{2}
\end{figure}
First let us introduce some notions. Let $p\in U$, where $U$ is any simply connected open subset of $M$, and let $H(U,p)$ be the holonomy group of the restriction $\nabla^E:\Gamma(E_{|U})\rightarrow \Gamma(T^* U\otimes E_{|U})$ in the point $p$. The space $\mathrm{Par}(E_{|U},\nabla^E)$ is isomorphic to the holonomy-invariant elements $P(U,p)=\{u\in E_p:hu=u\,\forall h\in H(U,p)\}$ in the fiber $E_p$, the isomorphism is given by parallel extension of $u\in P(U,p)$ to a parallel section on $U$. Since $U$ is simply connected, the group $H(U,p)$ is connected. Then, $P(U,p)$ coincides with 
\begin{equation} \{u\in E_p:hu=0\,\forall h\in\mathfrak{h}(U,p)\},
\end{equation}
where $\mathfrak{h}(U,p)$ is the Lie algebra of $H(U,p)$. Let $R^E\in \Gamma(\Lambda^2 T^*M\otimes\mathrm{End}(E))$ be the curvature of $\nabla^E$. By the theorem of Ambrose-Singer (see e.g. \cite{Kob}), $\mathfrak{h}(U,p)$ as a vector space is generated by elements of the form 
\begin{align}
(\tau_c)^{-1}R^E(X,Y)\tau_c:E_p\rightarrow E_p,\label{eq:generators}
\end{align}
where $X,Y\in T_q M$, $q\in M$, and  $\tau_c:E_{c(0)}\rightarrow E_{c(1)}$ is the parallel displacement along  a certain piece-wise smooth curve  $c:[0,1]\rightarrow U$ with $c(0)=p$ and $c(1)=q$. 

Since $\mathfrak{h}(M,p)$ is a finite-dimensional vector space  there  exist finitely many  curves $c_1,...,c_N:[0,1]\rightarrow M$  starting at  $p$  such that 
$\mathfrak{h}(M,p)$ as a vector space is generated by finitely many elements of the form 
\begin{align}
(\tau_{c_i})^{-1}R^E(X,Y)\tau_{c_i}:E_p\rightarrow E_p.\label{eq:generators1}
\end{align}
If we sligtly perturbe these curves, the corresponding elements \eqref{eq:generators1} will still generate $\mathfrak{h}(M,p)$ so we may assume that the curves have no intersections and self-intersections. Then, a sufficiently thin tubular neighborhood $U$ of the union of the curves $c_i$ is a ball, see  figure \ref{2}. 

The degree of mobility of the restriction of the Kähler structure to the ball $U$ clearly coincides with the degree of mobility of the Kähler structure on the whole manifold since the holonomy groups have the same algebras and therefore coincide. 
\end{proof}

\subsection{Proof of the second ``realization'' part of Theorems \ref{thm:degree} and \ref{thm:degreeeinstein}}
\label{sec:realization}
Let $2n\geq 4$,  $k\in \{0,...,n-1\}$ and $\ell\in\{1,...,[\frac{n+1-l}{2}]\}$. We need to construct a  $2n$-dimensional simply-connected Riemannian  K\"ahler  manifold  $(M,g,J)$
such that $D(g,J)=k^2+\ell$  and such that in the case $k^2+\ell\ge 2$ there exists a metric $\tilde g$  that is c-projectively but not affinely  equivalent to $g$.

The construction is as follows: we consider the direct product 
\begin{align}
(\hat{M},\hat{g},\hat{J})=(M_0,g_0,J_0)\times(M_1,g_1,J_1)\times...\times(M_\ell,g_\ell,J_\ell)\label{eq:decompmhat}
\end{align}
of Riemannian  Kähler manifolds. The manifold $(M_0,g_0,J_0)$ is the standard $R^{2k}$ with the standard flat metric and the standard complex structure. The Riemannian Kähler  manifolds $(M_i,g_i,J_i)$  for $i\ge 1$ satisfy the following conditions: they have dimension $\ge 4$, admit no nontrivial parallel hermitian symmetric $(0,2)$-tensor field, are cone manifolds, and the sum of their dimensions is  $2(n+1- k)$. The existence of such  $(M_i,g_i,J_i)$ is trivial: because of the condition $\ell\in\{1,...,[\frac{n+1-l}{2}]\}$ there exists a decomposition of $2(n+1- k)$ in the sum of  the integer  numbers  $2k_1+...+2k_\ell$ such that every $k_i\ge 2$. Now, as the manifold $(M_i,g_i,J_i)$ we take the conification of the standard $(R^{2k_i-2}, g_{flat}, J_{standard})$. They are cone  manifolds and they admit no nontrivial parallel hermitian symmetric $(0,2)$-tensor fields since for example by Theorem \ref{thm:coneconstruction2} the existence of such a tensor field will imply that the constant $B$ of the standard $(R^{2k_i-2}, g_{flat}, J_{standard})$ is $B=-1$ though it is equal to zero.

Let us also note, in view of the proof of Theorem \ref{thm:coneconstruction1}, that $(M_0,g_0,J_0)=(\mathbb{R}^{2k}\setminus \{0\},g_{flat},J_{0})$ coincides (at least locally) 
with the conification of $(\mathbb{C}P(k-1),g_{FS},J_{standard})$ via the Hopf fibration $S^{2k-1}\rightarrow \mathbb{C}P(n)$.

The direct product \eqref{eq:decompmhat} is clearly a Riemannian Kähler manifold. By \cite[Lemma 5]{FedMat}, it is a cone manifold, so by Theorem \ref{thm:coneconstruction1} it is (at least in a neighborhood of almost every point) the  conificiation of a certain $2n$-dimensional Kähler manifold. This manifold has degree of mobility $k^2+\ell$ since the dimension of  parallel symmetric hermitian $(0,2)$ tensors on its conification (which is \eqref{eq:decompmhat}) is $k^2+\ell$ as we want. This completes the proof of Theorem \ref{thm:degree}.

Now, in order to construct a Kähler-Einstein metric with degree of mobility  $D(g,J)=k^2+\ell$ where $k\in \{0,...,n-2\}$ and $\ell\in\{1,...,[\frac{n+1-l}{3}]\}$ we proceed along the same lines of ideas used above but assume in addition that the manifolds $(M_i,g_i,J_i)$ are Ricci-flat (and as such manifolds we can take conifications of Kähler-Einstein manifolds with scalar curvature chosen in correspondence with section \ref{-4}) and irreducible. The restrictions $k\in \{0,...,n-2\}$ and $\ell\in\{1,...,[\frac{n+1-l}{3}]\}$ imply that this is possible.   This completes the proof of Theorem \ref{thm:degreeeinstein}.

\section{Proof of Theorems \ref{thm:hprotrafo} and \ref{thm:hprotrafoeinstein}}
\label{sec:thmhprotrafo}
We need to  show that  $\mathrm{dim}(\mathfrak{c}(g,J)/\mathfrak{i}(g,J))$ is given by one of the values in the list of Theorem \ref{thm:hprotrafo} for a generic metric or by one of the values in the list of Theorem \ref{thm:hprotrafoeinstein} if the metric is Einstein. We assume that the manifold is simply connected and that the Kähler metric is Riemannian. 
 
Denote by $\mathfrak{h}(g,J)$ the Lie algebra of homothetic vector fields of $(M,g,J)$, i.e., vector fields $v$ satisfying $\mathcal{L}_v g=\mbox{const}\cdot g$, where $\mathcal{L}_v$ denotes the Lie derivative with respect to $v$. Consider the following sequence
\begin{align}
0\rightarrow \mathfrak{h}(g,J)\hookrightarrow \mathfrak{c}(g,J)\overset{f}{\longrightarrow} \mathcal{A}(g,J)/\mathbb{R}g\rightarrow 0,\label{eq:sequence}
\end{align}
where the mapping $f$ is given by 
$$f(v)=-\frac{1}{2}\left(\mathcal{L}_v g-\frac{\mathrm{trace}(g^{-1}\mathcal{L}_v g)}{2(n+1)}g\right)\mbox{ mod }\mathbb{R}g.$$
It is straight-forward to check that the tensor contained in the brackets on the right-hand side is indeed a solution of \eqref{eq:mainA}, for a proof see \cite[Lemma 2]{MatRos}. From the formula for $f$, it is straight-forward to see that $f(v)=0 \mbox{ mod }\mathbb{R}g$ if and only if $v$ is a homothetic vector field. Thus, the kernel of $f$ coincides with the image of the inclusion map from $\mathfrak{h}(g,J)$ to $\mathfrak{c}(g,J)$ and the sequence \eqref{eq:sequence} is exact at the first two stages. In particular, it follows that 
\begin{align}
\mathrm{dim}(\mathfrak{c}(g,J)/\mathfrak{h}(g,J))\leq D(g,J)-1.\label{eq:inequality}
\end{align}

From this inequality we see that in the case $D(g,J)=2$ the codimension of $\mathfrak{i}(g,J)$ in $\mathfrak{c}(g,J)$ is at most equal to one so  $\dim(\mathfrak{c}(g,J)/\mathfrak{i}(g,J))$ is $1$ or $0$. These two values are equal to  $\mathrm{dim}(\mathfrak{c}(g,J)/\mathfrak{i}(g,J))$ for certain $2n\ge 4$-dimensional  Riemannian K\"ahler-Einstein  manifolds  $(M, g,J)$ admitting a c-projectively equivalent metric which is not affinely equivalent to it. 
Indeed, most of the closed Riemannian Kähler-Einstein manifolds $(M,g,J)$ constructed in \cite{ApostolovII} that admit c-projectively equivalent metrics  which are not affinely equivalent to $g$, are of non-constant holomorphic  curvature and therefore admit no non-killing c-projective vector field by the Yano-Obata conjecture \cite{MatRos}. These examples have therefore $\dim(\mathfrak{c}(g,J)/\mathfrak{i}(g,J))=0$. 
In order to consider the case $D(g, J)\ge  3$, 
and also to construct examples with $\dim(\mathfrak{c}(g,J)/\mathfrak{i}(g,J))=1$, we need the following 
 
\begin{lem} \label{-5}   Let $(M,g,J)$ be a  connected Kähler  manifold of real dimension $2n\ge 4$ such that  the equation  \eqref{eq:hprosystem} admits a solution  $(A, \lambda, \mu)$  with $\lambda\ne 0$. Assume $B\ne 0$.  Then, 
\begin{align}
\mathrm{dim}(\mathfrak{c}(g,J)/\mathfrak{i}(g,J))= D(g,J)-1.\label{eq:Dhprovec}
\end{align}
\end{lem} 
\begin{proof}  
Without loss of generality we can assume that $B=-1$. Let us first show that $g$ admits no (local) homothety which is not an isometry. 
Suppose $F:M\rightarrow M$ is a homothety for $g$, i.e. $F^* g=cg$ for a certain constant $c$. For the new metric $\tilde g= cg$, the system \eqref{eq:hprosystem} holds for the  constant $\tilde B= B(cg)=\tfrac{1}{c}B$. But the constant $B$ is unique: since $cg$ and $g$ are isometric via $F$, the constants $B$ and $B(cg)$ (i.e., the constant $B$ corresponding to the metric $cg$)  must coincide. It follows that $c=1$ and consequently every homothety is an isometry.

In view of this, the sequence \eqref{eq:sequence} reads 
\begin{align}
0\rightarrow \mathfrak{i}(g,J)\hookrightarrow \mathfrak{c}(g,J)\overset{f}{\longrightarrow} \mathcal{A}(g,J)/\mathbb{R}g\rightarrow 0,\label{eq:sequence1}
\end{align}
Let us now  show that the sequence \eqref{eq:sequence} is exact which of course immediately implies the equality \eqref{eq:Dhprovec}.

In order to do it, we show the existence of a splitting 
$$h:\mathcal{A}(g,J)/\mathbb{R}g\rightarrow\mathfrak{c}(g,J)$$
of the sequence \eqref{eq:sequence}. The mapping $h$ is explicit and  sends a  solution $A$ to the corresponding vector field $\Lambda=g^{-1}\lambda$. Using the system \eqref{eq:hprosystem}, it is straight-forward to check that $\Lambda$ is a c-projective vector field (for an explicit proof see \cite[Proposition 10.3]{Tanno1978}). Moreover, the vector field $\Lambda$ is the same for $A$ and $A+\mbox{const}\cdot g$, hence, $h$ is well-defined and linear. For the composition of $f$ and $h$, we calculate
$$f( h (A\mbox{ mod }\mathbb{R}g))=f(\Lambda)=-\frac{1}{2}\left(2\nabla\lambda-\frac{\mathrm{trace}(\nabla\Lambda)}{n+1}g\right)\mbox{ mod }\mathbb{R}g$$
$$\overset{\eqref{eq:hprosystem}}{=}-\frac{1}{2}\left(\frac{2\mu g-2(n+1)A+\mathrm{trace}(A)g}{n+1}\right)\mbox{ mod }\mathbb{R}g.$$
The third equation in \eqref{eq:hprosystem} implies $d\mu=-2\lambda=-\tfrac{1}{2}d\mathrm{trace}A$. Thus, the functions $2\mu$ and $-\mathrm{trace}A$ coincide up to adding a constant. This shows that 
$$f( h (A\mbox{ mod }\mathbb{R}g))=A \mbox{ mod }\mathbb{R}g,$$
implying that \eqref{eq:sequence} is a splitting exact sequence and  \eqref{eq:Dhprovec} holds. 
\end{proof}

\begin{rem} 
The image of the map $h$ is precisely the ``canonical'' space of essential c-projective vector fields whose existence we announced in the introduction.  
\end{rem} 

Combining Lemma \ref{-5} with Theorem \ref{thm:degree}, we obtain the list from Theorem \ref{thm:hprotrafo}. Combining Lemma \ref{-5} with Theorem \ref{thm:degreeeinstein}, we obtain the list from Theorem \ref{thm:hprotrafoeinstein} (under the additional assumption that $B\neq0$). 

Note that by using Lemma \ref{-5}, also the values from the lists of Theorem \ref{thm:hprotrafo} or Theorem \ref{thm:degreeeinstein} can be obtained  as the number $\mathrm{dim}(\mathfrak{c}(g,J)/\mathfrak{i}(g,J))$ since in Section \ref{sec:realization} we also constructed metrics in all considered dimensions admitting solutions $(A, \lambda, \mu)$ of \eqref{eq:hprosystem} with $\lambda\ne 0$ such that their degree of mobility is two. The only case that cannot be constructed in this way is a $4$-dimensional K\"ahler-Einstein structure $(g,J)$ with $\mathrm{dim}(\mathfrak{c}(g,J)/\mathfrak{i}(g,J))=1$ since we can only produce examples of constant holomorphic  curvature by using the procedure from Section \ref{sec:realization}. We therefore construct an explicit example:  consider the local $4$-dimensional Riemannian K\"ahler structure $(g,\omega,J=-g^{-1}\omega)$ given in coordinates $x,y,s,t$ by  
\begin{align}
\begin{array}{c}
g=(x-y)( dx^2+ dy^2)+\frac{1}{x-y}\left[\left(ds+x dt\right)^2+\left(ds+ y dt\right)^2\right],\vspace{1mm}\\
\hat{\omega}=dx\wedge(ds+y dt) +dy\wedge( ds+x dt).
\end{array}\nonumber
\end{align}

This K\"ahler structure is a special case of those obtained in \cite{ApostolovI,BMMR}. It is straight-forward to check that $g$ is Ricci-flat but non-flat and that the $(1,1)$-tensor
\begin{align}
A=x\,\partial_x\otimes dx+y\,\partial_y\otimes dy+(x+y)\partial_s \otimes ds+xy\,\partial_s\otimes dt-\partial_t\otimes ds
\nonumber
\end{align}
is contained in $\mathcal{A}(g,J)$ (when viewed as $(0,2)$-tensor) and is non-parallel (and thus, corresponds to a K\"ahler metric $\tilde{g}$, that is c-projectively equivalent to $g$ and not affinely equivalent). Moreover, the vector field $v=x\,\partial_{x}+y\,\partial_{y}+2s\,\partial_s+t\,\partial_t$ is a c-projective vector field for $g$ and it is not Killing (thought, it is in fact an infinitesimal homothety, i.e. we have $\mathcal{L}_v g=3g$). For this metric we have $\dim(\mathfrak{c}(g,J)/\mathfrak{i}(g,J))=1$.

Let us now consider the case $B=0$. In this situation, as in the proof of Theorem \ref{thm:degree}, we change the metric in the c-projective class to make $B$ non-zero. This can be done on every open connected subset with compact closure, see Corollary \ref{cor:changeofmetric} and Corollary \ref{cor:changeofmetriceinstein}. The next lemma shows that the number $\mathrm{dim}(\mathfrak{c}(g,J)/\mathfrak{i}(g,J))$ remains the same. 
\begin{lem}[follows from \cite{EMN}]
Let $(M,g,J)$ be a simply connected K\"ahler manifold of real dimension $2n\geq 4$. Then, $\mathrm{dim}\,\mathfrak{i}(g,J)=\mathrm{dim}\,\mathfrak{i}(\tilde{g},J)$ for any metric $\tilde{g}$ that is c-projectively equivalent to $g$. 
\end{lem}
\begin{proof} 
We give a shorter version of the proof from \cite{EMN}. Let $K$ be a Killing vector field for $(g,J)$. It follows  that $K$ is also symplectic for the K\"ahler $2$-form $\omega=g(.,J.)$. Since we are working on a simply connected space, every symplectic vector field arises from a hamiltonian function $f$, i.e. $K=X_f$, where $X_f$ is defined by $g(X_f,J.)=df$. The condition that $K$ is Killing is equivalent to  the condition that  $\nabla\nabla f$ is hermitian.

Let $\phi$ be the function given by \eqref{M2} and consider the function $\mathrm{e}^{2\phi} f$. Let us show that this function is the hamiltonian  function for a Killing vector field for $\tilde{g}$. The geometry behind this statement is explained in \cite{EMN}.
  
We need to show that the symmetric $(0,2)$-tensor field $\tilde{\nabla}\tilde{\nabla}(e^{2\phi}f)$ is hermitian. 

First of all, it is  well-known, see  for example \cite[Proposition 3]{ApostolovI} or \cite[Lemma 3.2]{Kiyohara2010}, that the function $\mathrm{e}^{-2\phi}$ is the hamiltonian for a Killing vector field for $g$ and,  swapping the metrics $g$ and $\tilde g$, that  $\mathrm{e}^{2\phi}$ is the hamiltonian  function for a Killing vector field for $\tilde{g}$. Consequently, its hessian $\tilde{\nabla}\tilde{\nabla}(e^{2\phi})$ is hermitian. 

Using the transformation law \eqref{eq:lchpro}, we calculate 
$$\tilde{\nabla}\tilde{\nabla}(e^{2\phi}f)=f(\tilde{\nabla}\tilde{\nabla}e^{2\phi})+e^{2\phi}(\tilde{\nabla} df+2\Phi\odot df)$$
$$=f(\tilde{\nabla}\tilde{\nabla}e^{2\phi})+e^{2\phi}(\nabla df+\Phi\odot df+\Phi(J.)\odot df(J.)).$$

We see that the right-hand side of the above equation is hermitian, thus,  the Hamiltonian vector field of $e^{2\phi}f$ is a Killing vector field for $\tilde g$. If we 
 choose another hamiltonian function $f+\mbox{const}$ for $K$, the mapping 
$$K\longmapsto \tilde{K}=\tilde{X}_{e^{2\phi}f},$$
where $\tilde{X}_f$ is defined by $\tilde{g}(\tilde{X}_f,J.)=df$, is only defined up to adding constant multiples of $X_{e^{2\phi}}$. Thus,  $\mathrm{dim}\,\mathfrak{i}(g,J)$ coincides with $\mathrm{dim}\,\mathfrak{i}(\tilde{g},J)$ as we claimed. 
\end{proof}

Since obviously $\mathfrak{c}(g,J)=\mathfrak{c}(\tilde{g},J)$, it follows from the lemma that on each open simply connected neighborhood $U$, the number $\mathrm{dim}(\mathfrak{c}(g,J)/\mathfrak{i}(g,J))$ does not depend on the choice of the metric in the c-projective class. Suppose in addition that $U$ has compact closure. Then, by  Corollary \ref{cor:changeofmetric} there exists a Riemannian  metric $\tilde{g}$ on $U$ which is c-projectively equivalent to $g$ and such that the system \eqref{eq:hprosystem} for $\tilde{g}$ holds with a constant $\tilde{B}=-1$. Thus, by the already proven part, when restricted to a simply connected open subset $U$ with compact closure, the number $\mathrm{dim}(\mathfrak{c}(g,J)/\mathfrak{i}(g,J))$ is given by one of the values from the list of Theorem \ref{thm:hprotrafo} and, in the Einstein situation, it is given by one of the values from the list of Theorem \ref{thm:hprotrafoeinstein}.

In order to prove Theorem  \ref{thm:hprotrafo} and \ref{thm:hprotrafoeinstein} on the whole manifold, we again  use Lemma \ref{lem:extensiontoglobal}. It is known that Killing vector fields could be viewed as parallel sections of a certain  vector bundle.  The same is true for  c-projective vector fields, for example because c-projective geometry is a parabolic geometry, see for example \cite[Section 3.3]{CapGoverHammerl}, \cite[Section 4.6]{cap} or \cite{hrdina}, and infinitesimal symmetries of parabolic geometries are sections of a certain vector bundle. Actually, in \cite{EMN} the vector bundle and also the connection on it are explicitly constructed. By Lemma \ref{lem:extensiontoglobal}, the number $\mathrm{dim}(\mathfrak{c}(g,J)/\mathfrak{i}(g,J))$ on the whole manifold is the same as this number for the restriction of the K\"ahler structure to a certain ball and above we have shown that this value is contained in the list of Theorem \ref{thm:hprotrafo} or, in the Einstein situation, in the list of  Theorem \ref{thm:hprotrafoeinstein}.

\section{Proof of Theorem \ref{thm:einstein}}\label{sec:thmeinstein}

Let us first recall the following statement from \cite{haddad2} (and give a full proof since the publication  is not easy to find) 
\begin{lem}\label{lem:einsteincproeinstein2}
Let $g,\tilde{g}$ be c-projectively equivalent K\"ahler-Einstein metrics on the connected complex manifold $(M,J)$ of real dimension $2n\geq 4$. Then, for every $A\in \mathcal{A}(g,J)$ with corresponding $1$-form $\lambda$, there exists a function $\mu$ such that $(A,\lambda,\mu)$ satisfies \eqref{eq:hprosystem} with $B=-\frac{\mathrm{Scal}(g)}{4n(n+1)}$.
\end{lem}
\begin{proof}
Denote by $A=A(g,\tilde{g})$ the solution of \eqref{eq:mainA} given by \eqref{eq:defA}. We can insert the relation \eqref{eq:relationphilambda} between the $1$-forms $\Phi$ in \eqref{eq:lchpro} and $\lambda$ in \eqref{eq:mainA} into \eqref{eq:riccitafo} to obtain the change of the Ricci-tensors in terms of $A$ and $\lambda$. Denoting by $\Lambda=g^{-1}\lambda$ the vector field corresponding to $\lambda$, a straight-forward calculation shows
$$\mathrm{Ric}(\tilde{g})=\mathrm{Ric}(g)+2(n+1)(g(A^{-1}\nabla\Lambda.,.)-g(A^{-1}\Lambda,\Lambda)g(A^{-1}.,.)).$$
Now suppose both metrics are Einstein, that is $\mathrm{Ric}(\tilde{g})=\tilde{c}\tilde{g}$ and $\mathrm{Ric}(g)=cg$ for constants $c=\tfrac{\mathrm{Scal}(g)}{2n},\tilde{c}=\tfrac{\mathrm{Scal}(\tilde{g})}{2n}$. Inserting this into the last equation and multiplying with $g^{-1}$ from the left yields 
$$\tilde{c}g^{-1}\tilde{g}=c\mathrm{Id}+2(n+1)(A^{-1}\nabla\Lambda-g(A^{-1}\Lambda,\Lambda)A^{-1}).$$
By \eqref{eq:defA}, $\tilde{g}$ can be written as $\tilde{g}=(\mathrm{det}A)^{-\tfrac{1}{2}}gA^{-1}$. Inserting this into the last equation and multiplying with $A$ from the left, we obtain
$$\tilde{c}(\mathrm{det}A)^{-\tfrac{1}{2}}\mathrm{Id}=cA+2(n+1)(\nabla\Lambda-g(A^{-1}\Lambda,\Lambda)\mathrm{Id}).$$
Rearranging terms yields
\begin{align}
\nabla\Lambda=\mu \mathrm{Id}+BA,\label{eq:secondeqsystem}
\end{align}
where we defined 
\begin{align}
\mu=\frac{\bar{c}(\mathrm{det}A)^{-\tfrac{1}{2}}}{2(n+1)}+g(A^{-1}\Lambda,\Lambda)\mbox{ and }B=-\frac{c}{2(n+1)}.\label{eq:functionandconstant}
\end{align}
Equation \eqref{eq:secondeqsystem} is exactly the second equation in \eqref{eq:hprosystem}. It remains to show that the third equation on the function $\mu$ is satisfied as well. In \cite[Remark 5]{FKMR} it was noted that if the second equation in the system \eqref{eq:hprosystem} holds for $B$ equal to a constant, the third equation in \eqref{eq:hprosystem} is satisfied automatically. This is sufficient for our purposes, however, we show that the third equation can be obtained directly by taking the covariant derivative of \eqref{eq:secondeqsystem}. We obtain
$$\nabla_X \nabla\Lambda=(\nabla_X\mu) \mathrm{Id}+B\nabla_X A\overset{\eqref{eq:mainA}}{=}(\nabla_X\mu) \mathrm{Id}+B(g(.,X)\Lambda+g(.,\Lambda)X+g(.,JX)J\Lambda+g(.,J\Lambda)JX).$$
Taking the trace of this equation yields
\begin{align}
\mathrm{trace}(\nabla_X \nabla\Lambda)=2n\nabla_X\mu+4Bg(X,\Lambda).\label{eq:thirdeqderiv}
\end{align}
As in the proof of Lemma \ref{lem:scalarB}, we can use that $J\Lambda$ is Killing to obtain the identity $\nabla_X\nabla \Lambda=-JR(X,J\Lambda)$. Together with the usual identities for the Ricci-tensor of a K\"ahler metric, this yields
$$\mathrm{trace}(\nabla_X \nabla\Lambda)=-\mathrm{trace}(J R(X,J\Lambda))=-2\mathrm{Ric}(X,\Lambda)=-2cg(X,\Lambda),$$
where we used the Einstein condition in the last step. Inserting the above formula and $B$ from \eqref{eq:functionandconstant} into \eqref{eq:thirdeqderiv}, we obtain the third equation in \eqref{eq:hprosystem}. 

We have shown that when $g,\tilde{g}$ are c-projectively equivalent K\"ahler-Einstein metrics, there exists a function $\mu$ and a constant $B$ such that the triple $(A=A(g,\tilde{g}),\lambda,\mu)$ satisfies \eqref{eq:hprosystem} for the metric $g$ (of course, by interchanging the roles of $g,\tilde{g}$ this can be obtained also in terms of $\tilde{g}$). Since in the case $D(g,J)=2$ every $A'\in \mathcal{A}(g,J)$ is a linear combination of $\mathrm{Id}$ and the solution $A$ from above, we obtain the proof of Lemma \ref{lem:einsteincproeinstein2} for $D(g,J)=2$. On the other hand, in the case $D(g,J)\geq 3$, Lemma \ref{lem:einsteincproeinstein2} follows as a direct application of Theorem \ref{thm:hprosystem} above. 
\end{proof}

\begin{rem}  Combining Lemmas \ref{lem:einsteincproeinstein} and \ref{lem:einsteincproeinstein2}, we obtain Theorem \ref{+2}.  
\end{rem} 

Let us  now prove  Theorem \ref{thm:einstein}, that is, let us show that two c-projectively equivalent K\"ahler-Einstein metrics on a closed connected complex manifold have constant holomorphic curvature unless they are affinely equivalent.

We have shown that the triple $(A,\lambda,\mu)$, where $A$ is the tensor from \eqref{eq:defA} constructed by the two metrics and $\lambda$ is the corresponding $1$-form from \eqref{eq:mainA}, satisfies the system \eqref{eq:hprosystem} for a certain constant $B$. If this constant is zero, we see from \eqref{eq:hprosystem} that the function $\mu$ is constant and $\nabla\lambda$ is parallel. Since $\lambda$ is the differential of a function and the manifold is closed, there are points where $\nabla\lambda$ is positively and negatively definite respectively (corresponding to the minimum and maximum value respectively of the function). Since $\nabla\lambda$ is parallel, it actually has to vanish identically, thus, $\lambda$ is parallel. Since it vanishes at points where the corresponding function is maximal, $\lambda$ has to be identically zero. Using Remark \ref{rem:affine}, this implies that the two Einstein metrics are affinely equivalent.

Now suppose that the metrics are not affinely equivalent. In particular, $B$ is not zero and the function $\mu$ is not constant. Using the equations from the system \eqref{eq:hprosystem}, we can succesively replace the covariant derivatives of $A$ and $\lambda$ to obtain that $\mu$ satisfies the third order system 
\begin{align}
\begin{array}{c}
(\nabla\nabla\nabla \mu)(X,Y,Z)=B[2(\nabla_{X}\mu)g(Y,Z)+(\nabla_{Z}\mu)g(X,Y)+(\nabla_{Y}\mu)g(X,Z)\vspace{2mm}\\-(\nabla_{JZ}\mu)g(JX,Y)-(\nabla_{JY}\mu)g(JX,Z)]
\end{array}.\label{eq:tanno}
\end{align}
of partial differential equations. This equation was studied in \cite{HiramatuK,Tanno1978}. There it was shown that the existence of non-constant solutions of this equation on a closed connected K\"ahler manifold implies that $B<0$ and the metric $g$ has constant holomorphic  curvature equal to $-4B$. By interchanging the roles of $g$ and $\tilde{g}$, this statement holds for $\tilde{g}$ as well. This completes the proof of Theorem \ref{thm:einstein}.

\subsection*{\bf Acknowledgements.} We are grateful to D. V. Alekseevsky, D. Calderbank,  M. Eastwood,  A. Ghigi, V. Kiosak  and C. T\o nnesen-Friedman for discussions and useful comments to this paper. Also, we thank    Deutsche Forschungsgemeinschaft (Research training group   1523 --- Quantum and Gravitational Fields) and FSU Jena for partial financial support.

\nocite{*}
\bibliographystyle{plain}

\begin{thebibliography}{99}


\bibitem{ACM} D. V. Alekseevsky, V. Cortes, T. Mohaupt, \emph{Conification of K\"ahler and hyper-K\"ahler manifolds}, accepted to Comm. Math. Phys.,    arXiv:1205.2964, 2012


\bibitem{Apostolov0} V. Apostolov, D. Calderbank, P. Gauduchon, \emph{The geometry of weakly self-dual Kähler surfaces}, Compositio Math.,  {\bf 135} ,  no. 3, 279--322, 2003, MR1956815  

\bibitem{ApostolovI} V. Apostolov, D. Calderbank, P. Gauduchon, 
\emph{Hamiltonian 2-forms in K\"ahler geometry. I. General theory}, J. Differential Geom.  {\bf 73},  no. 3, 359--412, 2006

\bibitem{ApostolovII} V. Apostolov, D. Calderbank, P. Gauduchon, C. T\o nnesen-Friedman, 
\emph{Hamiltonian 2-forms in K\"ahler geometry. II. Global classification},  J. Differential Geom. {\bf 68},  no. 2, 277--345, 2004

\bibitem{ApostolovIII} V. Apostolov, D. Calderbank, P. Gauduchon, C. T\o nnesen-Friedman \emph{Hamiltonian 2-forms in K\"ahler geometry. III. Extremal metrics and stability},
Invent. Math.  {\bf 173},  no. 3, 547--601, 2008

\bibitem{ApostolovIV} V. Apostolov, D. Calderbank, P. Gauduchon, C. T\o nnesen-Friedman, \emph{Hamiltonian 2-forms in K\"ahler geometry. IV. Weakly Bochner-flat K\"ahler manifolds},  Comm. Anal. Geom.  {\bf 16}, no. 1, 91--126, 2008


\bibitem{bandomabuchi} S. Bando, T. Mabuchi, \emph{Uniqueness of Einstein K\"ahler metrics modulo connected group actions}, Algebraic geometry, Sendai, 1985, 11--40, 
Adv. Stud. Pure Math., 10, North-Holland, Amsterdam, 1987, MR0946233 

\bibitem{besse} A. Besse, \emph{Einstein manifolds}, Springer, 1987

 
 
 
 
 \bibitem{BMMR} A. V. Bolsinov, V. S.  Matveev, T. Mettler, S. Rosemann, \emph{ Four-dimensional K\"ahler metrics  admitting  essential c-projective vector fields,} in preparation.  
  
 
\bibitem{cap} A. \v{C}ap, \emph{Correspondence spaces and twistor spaces for parabolic geometries}, J. Reine Angew. Math. {\bf 582}, 143--172, 2005,  MR2139714

\bibitem{CapGoverHammerl} A. \v{C}ap, A. R. Gover, M. Hammerl, \emph{Holonomy reductions of Cartan geometries and curved orbit decompositions}, arXiv:1103.4497 [math.DG], 2011

\bibitem{DR} G. De Rham, {\em Sur la reductibilit\'e d'un espace de
    Riemann, } Comment. Math. Helv. {\bf 26}, 328--344, 1952.

\bibitem{EMN}
M. Eastwood, V. Matveev, K.  Neusser, \emph{C-projective geometry: background and  open problems }, in preparation. 
\bibitem{FKMR} A. Fedorova, V. Kiosak, V. Matveev, S. Rosemann, {\it The only K\"ahler manifold with degree of mobility at least 3 is $(CP(n), g_{Fubini-Study})$}, Proc. London Math. Soc., {\bf  105},  no. 1, 153--188, 2012, 
 doi: 10.1112/plms/pdr053, 2012

\bibitem{FedMat} A. Fedorova, V. Matveev, \emph{Degree of mobility for metrics of lorentzian signature and parallel (0,2)-tensor fields on cone manifolds}, arXiv:1212.5807 [math.DG], 2012


\bibitem{HiramatuK} H. Hiramatu, {\it Integral inequalities in K\"ahlerian manifolds and their applications}, Period. Math. Hungar.
  {\bf 12}, no. 1, 37--47, 1981,  MR0607627, Zbl 0427.53032. 

\bibitem{Hitchin} N. J. Hitchin, A. Karlhede, U. Lindström, M. Rocek, \emph{Hyper-K\"ahler metrics and supersymmetry}, Com. Math. Phys. 108 (4)  535--589, 1987, MR877637  
   
\bibitem{hrdina} J. Hrdina, \emph{Almost complex projective structures and their morphisms}, Arch. Math. (Brno) {\bf 45}, no. 4, 255--264,  2009, MR2591680




\bibitem{Kiyo1997} K. Kiyohara, \emph{Two classes of Riemannian
    manifolds whose geodesic flows are integrable},
  Mem. Amer. Math. Soc. {\bf 130}, no. 619, viii+143 pp., 1997

\bibitem{Kiyohara2010} K. Kiyohara, P. J.  Topalov, \emph{On Liouville
    integrability of h-projectively equivalent K\"ahler metrics},
  Proc. Amer. Math. Soc. {\bf 139}, 231--242, 2011. 


\bibitem{haddad1} V. Kiosak, M. Haddad, \emph{  On A-harmonic Kähler spaces,}   Geometry of generalized spaces, 41--45, Penz. Gos. Ped. Inst., Penza, 1992. 

\bibitem{haddad2} V. Kiosak, M. Haddad, \emph{  On  holomorphic-projective transformations of A-harmonic Kähler spaces,}   preprinted in Ukr. NIINTI  20.08.1991  no. 1217-UK91. 



\bibitem{Kob} S. Kobayashi, K. Nomizu, \emph{Foundations of
    Differential Geometry II}, John Wiley and Sons, Inc., 1996.


   
 






\bibitem{Matveev2010} V. S. Matveev, \emph{Gallot-Tanno theorem for
   pseudo-Riemannian manifolds and a proof that decomposable cones
    over closed complete pseudo-Riemannian manifolds do not exist},
  J. Diff. Geom. Appl. {\bf 28}, no. 2, 236--240, 2010

\bibitem{Mounoud2010} V. S. Matveev, P. Mounoud,
  \emph{Gallot-Tanno Theorem for closed incomplete pseudo-Riemannian
    manifolds and applications},  Ann. Glob. Anal. Geom. {\bf 38}, 259--271, 2010



\bibitem{hall}  V. S. Matveev, {\it Geodesically equivalent metrics in general relativity,} J. Geom. Phys. {\bf  62}, no.  3,  675--691,  2012. 

\bibitem{MatRos} V. S. Matveev, S. Rosemann, \emph{Proof of the Yano-Obata conjecture for h-projective transformations}, J. Differential Geom. {\bf 92}, no. 1, 221--261, 2012

\bibitem{DomMik1978} J. Mikes, V. V. Domashev, \emph{On The Theory Of
    Holomorphically Projective Mappings Of Kaehlerian Spaces},
  Math. Zametki {\bf 23}, no. 2, 297--303, 1978




\bibitem{Mikes} J. Mikes, \emph{Holomorphically projective mappings
    and their generalizations.}, J. Math. Sci. (New
  York) {\bf 89}, no. 3, 1334--1353, 1998



\bibitem{moroianu} A. Moroianu, \emph{Lecture Notes on K\"ahler geometry}, http://www.math.polytechnique.fr/$\sim$moroianu/tex/kg.pdf

\bibitem{moruianusemmelmann} A. Moroianu, U. Semmelmann, \emph{Twistor forms on Kähler manifolds}, Ann. Sc. Norm. Super. Pisa Cl. Sci. (5) {\bf 2},    no. 4, 823--845, 
 2003, MR2040645


\bibitem{Otsuki1954} T. Otsuki, Y. Tashiro, \emph{On curves in
    Kaehlerian spaces}, Math. Journal of Okayama University {\bf 4}, 57--78, 1954




\bibitem{semmelmann} U. Semmelmann, \emph{Conformal Killing forms on Riemannian manifolds}, Math. Z.  {\bf 245},  no. 3, 503--527, 2003, MR2021568 


\bibitem{Sinjukov} N. S. Sinjukov, {\it Geodesic mappings of
    Riemannian spaces.} (in Russian) ``Nauka'', Moscow, 1979,
  MR0552022, Zbl 0637.53020.
  
\bibitem{Sinjukov2} N. S. Sinyukov, E. N. Sinyukova, \emph{Holomorphically projective mappings of special Kählerian spaces}, (Russian) 
Mat. Zametki  {\bf 36},    no. 3, 417--423, 1984, MR0767221
  

\bibitem{Tanno1978} S. Tanno, \emph{Some Differential Equations On
    Riemannian Manifolds}, J. Math. Soc. Japan {\bf 30}, no. 3, 509--531, 1978

\bibitem{Tashiro1956} Y. Tashiro, \emph{On A Holomorphically
    Projective Correspondence In An Almost Complex Space},
  Math. Journal of Okayama University {\bf 6}, 147--152, 1956





\bibitem{Wu} H.  
Wu,  \emph{On the de Rham decomposition theorem,} Illinois J. Math. {\bf 8},  291--311, 1964.
    
\bibitem{Yanobook} K. Yano, {\it Differential geometry on complex and
    almost complex spaces.}  International Series of Monographs in
  Pure and Applied Mathematics, Vol.  {\bf 49}, A Pergamon Press
  Book. The Macmillan Co., New York 1965 xii+326 pp.
  
  


\end{thebibliography}

\end{document}